\documentclass[a4paper, 10pt]{article}

\usepackage[cp1251]{inputenc}
\usepackage[english]{babel}
\usepackage{amsthm}
\usepackage{cite}
\usepackage{tikz-cd}
\usepackage{amsmath}
\usepackage{amsfonts}
\usepackage{amssymb}
\usepackage{hyperref}
\usepackage[]{authblk}
\usepackage{epsfig,graphicx}

\DeclareMathOperator{\Aut}{Aut}

\DeclareMathOperator{\Char}{char}
\DeclareMathOperator{\Det}{det}

\DeclareMathOperator{\Fr}{Fr}

\DeclareMathOperator{\Deg}{deg}

\DeclareMathOperator{\Sp}{Sp}
\DeclareMathOperator{\Id}{Id}
\DeclareMathOperator{\Ht}{ht}
\DeclareMathOperator{\Rk}{rank}
\DeclareMathOperator{\J}{J}

\DeclareMathOperator{\GL}{GL}
\DeclareMathOperator{\Ind}{Ind}
\newcommand{\TAut}{\operatorname{TAut}}

\newtheorem{thm}{Theorem}[section]
\newtheorem{lem}[thm]{Lemma}

\newtheorem{prop}[thm]{Proposition}
\newtheorem{cor}[thm]{Corollary}
\newtheorem{conj}[thm]{Conjecture}
\newtheorem{Def}[thm]{Definition}
\newtheorem{remark}[thm]{Remark}
\begin{document}
\fontsize{11}{11pt}\selectfont
\renewcommand{\thefootnote}{\fnsymbol{footnote}}
\renewcommand{\thefootnote}{\arabic{footnote}}
\fontsize{12}{12pt}\selectfont
\title{\bf Augmented Polynomial Symplectomorphisms and Quantization}
\renewcommand\Affilfont{\itshape\small}

\author[1]{Alexei Belov-Kanel\thanks{kanel@mccme.ru}}
\author[2,3]{Andrey Elishev\thanks{elishev@phystech.edu}}
\author[1]{Jie-Tai Yu\thanks{jietaiyu@szu.edu.cn}}

\affil[1]{College of Mathematics and Statistics, Shenzhen University, Shenzhen, 518061, China}
\affil[2]{Laboratory of Advanced Combinatorics and Network Applications, Moscow Institute of Physics and Technology, Dolgoprudny, Moscow Region, 141700, Russia}
\affil[3]{Department of Innovations and High Technology, Moscow Institute of Physics and Technology, Dolgoprudny, Moscow Region, 141700, Russia}

\date{}

\maketitle
\renewcommand{\abstractname}{Abstract}
\begin{abstract}
The objective of this paper is the proof of a conjecture of Kontsevich \cite{BKK1} on the isomorphism between groups of polynomial symplectomorphisms and automorphisms of the corresponding Weyl algebra in characteristic zero. The proof is based on the study of topological properties of automorphism $\Ind$-varieties of the so-called augmented and skew augmented versions of Poisson and Weyl algebras. Approximation by tame automorphisms as well as a certain singularity analysis procedure is utilized in the construction of the lifting of augmented polynomial symplectomorphisms, after which specialization of the augmentation parameter is performed in order to obtain the main result.
\end{abstract}

\section{Introduction}

Let $W_{n, \mathbb{K}}$ denote the $n$-th Weyl algebra over a field $\mathbb{K}$, which is by definition the quotient of the free associative algebra
\begin{equation*}
\mathbb{K}\langle a_1,\ldots,a_n,b_1,\ldots,b_n\rangle
\end{equation*}
by the two-sided ideal generated by elements
\begin{equation*}
b_ia_j-a_jb_i-\delta_{ij},\;\;a_ia_j-a_ja_i,\;\;b_ib_j-b_jb_i.
\end{equation*}
Let also $P_{n,\mathbb{K}}$ denote the commutative polynomial algebra $\mathbb{K}[x_1,\ldots,x_{2n}]$ equipped with the standard Poisson bracket
\begin{equation*}
\lbrace x_i,x_j\rbrace = \omega_{ij} \equiv \delta_{i,n+j}-\delta_{i+n,j}
\end{equation*}
($\delta_{ij}$ is the Kronecker symbol). We will call $P_{n,\mathbb{K}}$ the (commutative) \textbf{Poisson algebra}. The Poisson algebra $P_n$ is the classical counterpart of the Weyl algebra $W_n$ with respect to the deformation quantization. It is known that in the case of the polynomial algebra, the deformation quantization admits a natural reverse procedure, which we will call anti-quantization, that allows one to construct classical objects (that is, objects in $P_n$ or those defined by subsets of $P_n$) from quantized ones. The correspondence takes a rather elaborate form in characteristic zero.

This paper is dedicated to the study of the following conjecture, proposed by Kontsevich.
\begin{conj} \label{mainconj}
Whenever $\Char \mathbb{K} = 0$, one has the following (canonical) isomorphism
$$
\Aut W_{n, \mathbb{K}}\simeq \Aut P_{n, \mathbb{K}}.
$$
\end{conj}

The functor $\Aut$ returns the group of automorphisms of a $\mathbb{K}$-algebra. Henceforth we call the elements of $\Aut P_{n,\mathbb{K}}$ \textbf{polynomial symplectomorphisms}, keeping in mind the fact that this group is in natural one-to-one correspondence with the group of polynomial automorphisms of the affine space $\mathbb{A}^{2n}_{\mathbb{K}}$ which preserve the standard symplectic structure. Conjecture \ref{mainconj} then states that in characteristic zero, the group of Weyl algebra automorphisms is naturally isomorphic to the group of polynomial symplectomorphisms.

In the work \cite{BKK1}, various approaches to Conjecture \ref{mainconj} along with its generalizations were considered. In particular, a group homomorphism
\begin{equation*}
\Phi: \Aut W_{n,\mathbb{C}}\rightarrow \Aut P_{n,\mathbb{C}}
\end{equation*}
was constructed and its properties were analyzed. The morphism $\Phi$ is a candidate for the isomorphism between the automorphism groups over $\mathbb{C}$ and is in fact identical on the subgroups of so-called \textbf{tame automorphisms}. More precisely, a theorem established in \cite{BKK1} asserts that the morphism $\Phi$ induces an isomorphism of tame subgroups
\begin{equation*}
\Phi:\TAut W_{n,\mathbb{C}}\xrightarrow{\sim} \TAut P_{n,\mathbb{C}}.
\end{equation*}

The construction of $\Phi$ goes back to Tsuchimoto \cite{Tsu2} and relies on reduction of the Weyl algebra to positive characteristic. The essence of the latter procedure, which we will describe in the next section along with the definition of $\Phi$, is the representation of the ground field $\mathbb{C}$ as a reduced direct product, modulo a fixed non-principal ultrafilter $\mathcal{U}$ on the index set, of algebraically closed fields $\mathbb{F}_p$ of positive characteristic $p$
$$
\mathbb{C}\simeq \left(\prod_{p} \mathbb{F}_p\right) /\;\mathcal{U}
$$
with $p$ running along an $\mathcal{U}$-unbounded sequence of primes. This rather unusual technique allows one to use the fact that in characteristic $p$ the algebra $$W_{n,\mathbb{F}_p}\simeq \mathbb{F}_p[x_1,\ldots, x_n, d_1,\ldots, d_n]$$ is Azumaya over its center; the center itself is the (commutative) polynomial algebra generated by the $p$-th powers of algebra generators:
$$
\mathbb{F}_p[x_1^p,\ldots, x_n^p, d_1^p,\ldots, d_n^p]
$$
and -- crucially -- possesses a Poisson bracket induced from the commutator in $W_n$. The endomorphisms of $W_n$ also preserve the center and so can be restricted to this Poisson algebra to produce symplectic polynomial mappings; those can then be re-assembled from the ultraproduct decomposition and returned to characteristic zero, and the procedure is manifestly homomorphic. Note that this construction (specifically the ultraproduct decomposition of $\mathbb{C}$) requires the ground field to be algebraically closed, which, along with the assumption of Conjecture \ref{mainconj} that $\Char\mathbb{K}=0$ allows one to set $\mathbb{K}=\mathbb{C}$ without loss of generality. The approach to Conjecture \ref{mainconj} involving the morphism $\Phi$, therefore, excludes the case of the rational numbers as the ground field.

\smallskip

The main result of the present paper is as follows.

\medskip

\begin{thm}[Main Theorem]\label{mainthm}
The homomorphism
$$
\Phi: \Aut W_{n,\mathbb{C}}\rightarrow \Aut P_{n,\mathbb{C}}
$$
defined previously is a group isomorphism.
\end{thm}
\medskip

\smallskip

In order to prove that $\Phi$ is an isomorphism, one may try and construct its inverse. To do so, one needs to find a way to \textbf{lift} polynomial symplectomorphisms to Weyl algebra automorphisms over $\mathbb{C}$. A viable approach to this lifting problem is made possible by the tame isomorphism property of $\Phi$ as stated above; indeed, tame symplectomorphisms are lifted unambiguously, therefore if one could find a suitable topology on $\Aut P_{n,\mathbb{C}}$, such that $\TAut P_{n,\mathbb{C}}$ were a dense subgroup, one would represent arbitrary polynomial symplectomorphisms as limits of sequences of tame symplectomorphisms and then make the limits of the pre-images under $\Phi$ of those sequences into automorphisms of $W_{n,\mathbb{C}}$. Such a topology does in fact exist and is the so-called formal power series topology (or, as we sometimes refer to it, augmentation topology) introduced in the classical work of Anick \cite{An}. A symplectic version of Anick's approximation results also holds, as we have shown in our recent work \cite{KGE}, so that the path to symplectomorphism lifting is seemingly clear.

\medskip

The lifting procedure done in this way, however, results in certain complications. The most obvious one is the fact that the lifted limit is an automorphism of the power series completion of $W_n$, so that in order to complete the lifting one must find a way to prove the polynomial character of such series. There is, however, no straightforward way of doing so -- certain topological properties which would otherwise guarantee canonicity of the procedure are missing. One way around this is the introduction of deformation (or augmentation) of the algebra and its Poisson structure by a central variable $h$, so that
$$
\lbrace p_i, x_j\rbrace = h\delta_{ij}
$$
in the new algebra. This augmentation modifies (homogenizes) the Poisson structure to match the formal power series topology which defines tame approximation. Furthermore, the Poisson structure can be further distorted by allowing non-zero commutation of distinct generators, in such a way as to make the morphisms in question continuous. The tradeoff of the augmentation approach is the need to \emph{specialize} the new variable $h$ (to $h = 1$) in order to return to the original conjecture for non-augmented algebras. This is achieved by the extension of the domain of the lifting map to points rational in the augmentation parameter. The procedure is {\bf not at all trivial } and requires the invertibility of mappings corresponding to points in $\Aut$ (that is, it does not work with endomorphisms).

\smallskip

In order to obtain the stronger topological properties of tame approximation (continuity of the loop morphism in the skew augmented case), we use a certain \textbf{singularity trick}. In the broadest terms, it is a proof technique that allows one, by examining orders of singularities of certain curves and their images under the studied morphism, to obtain useful data on the morphism. The technique was utilized in \cite{KBYu} as well as in the main text. It is not applicable in the non-augmented case.

\smallskip

In \cite{BKK1}, several generalizations of Conjecture \ref{mainconj} were studied. The most general reformulation has to do with holonomic modules over the Weyl algebra (holonomic $\mathcal{D}$-modules) and is stated as follows.
\begin{conj}\label{conjhol}
There is a one-to-one correspondence between irreducible holonomic $\mathcal{D}$-modules over $W_n$ and lagrangian subvarieties of the affine space (of corresponding dimension).
\end{conj}

One direction in the Conjecture \ref{conjhol} -- namely the construction of a lagrangian subvariety from a given holonomic module -- has been accomplished by Bitoun \cite{Bit} and, independently, Van den Bergh \cite{VdB}, who gave a conceptually different proof. Also, the one-dimensional case of Conjecture \ref{conjhol} was studied in \cite{K-BE}.


Dodd \cite{Dodd} has established a number of far-reaching results of homological nature which, as far as our understanding is, imply a version of Conjecture \ref{conjhol} (cf. Theorem 1, Corollary 2 and Theorem 3 of \cite{Dodd}). His argument is based on properties of the so-called $p$-support, defined by Kontsevich in \cite{Kon2}.

Our approach is different from Dodd's and focuses mainly of the topological properties of automorphism ind-varieties. This technique is in line with the philosophy of Shafarevich and his school.

\smallskip

Most of our analysis takes place over algebraically closed base field -- in particular, Tsuchimoto's construction of the homomorphism $\Phi$ leads to an algebraically closed "universal" base field of characteristic zero. However, in \cite{BKK1}, a slightly different, more general construction of a group homomorphism over the rationals is presented -- namely, it is proved in \cite{BKK1} that for $R$ a commutative ring there exists a (unique) group homomorphism
$$
\phi_R: \Aut W_{n,R}\rightarrow \Aut P_{n,R_{\infty}}
$$
where
\begin{equation*}
R_{\infty}=\lim_{\rightarrow}\left( \prod_{p} R'\otimes \mathbb{Z}/p\mathbb{Z}\;/\;\bigoplus_{p} R'\otimes \mathbb{Z}/p\mathbb{Z}\right),
\end{equation*}
is the reduction modulo infinite prime. The image  $\phi_R(f)$ is essentially (untwisted by the Frobenius morphism) a collection of restrictions to centers of Weyl algebras over $\mod\;p$ reductions of a finitely generated subring $R'\subset R$ over which the automorphism $f$ is defined. This construction is supposed to be the replacement of $\phi_{[p]}$ for the case of arbitrary base field of characteristic zero (or equivalently the field $\mathbb{Q}$ of rational numbers). In light of that, one has the following statement (Conjecture 3 of \cite{BKK1}):
\begin{conj}\label{conj3}
The image of $\phi_R$ belongs to
$$
\Aut P_{n,i(R)\otimes \mathbb{Q}}
$$
where $i:R\rightarrow R_{\infty}$ denotes the tautological inclusion.
\end{conj}
From that Conjecture a collection of constructible maps
$$
\phi_{n,N}: \Aut^{\leq N}W_{n,\mathbb{Q}}\rightarrow \Aut^{\leq N}P_{n,\mathbb{Q}}
$$
may be defined, as proved in \cite{BKK1}. These maps serve as generalizations of $\phi_{[p]}$ (dealt with in this paper) and constitute the conjectured canonical isomorphism of Conjecture \ref{mainconj}. What this means for the ultrafilter construction considered here is that \emph{a priori} it is not known whether $\Phi$ is defined over the rationals. However, the introduction of augmentation and augmented tame approximation should allow us to circumvent this obstacle. The line of reasoning is as follows.

\smallskip

Firstly, the isomorphism of tame subgroups (Theorem 1 of \cite{BKK1}) continues to be valid for the base field $\mathbb{Q}$. Secondly, both in augmented and non-augmented cases, tame approximation requires only characteristic zero to work, and therefore is valid over the rationals. In the augmented case, however, tame approximation is stronger in the sense that the tame isomorphism is continuous in a neighborhood of the identity map. Therefore, for Weyl $\mathbb{Q}$-algebra automorphisms sufficiently close to the identity one can take a well defined limit of the image under the tame isomorphism of an arbitrary converging tame sequence. Thus defined mapping can be extended to the whole space to yield a mapping which must coincide with the augmented version of $\phi_R$ (for $R=\mathbb{Q}$). One then specializes the augmentation parameters to return to the non-augmented case. Most of the mechanics of the specialization, as well as the proof of the Main Theorem, is adapted to this situation. Thus one should be able to obtain the proof of Conjecture \ref{mainconj} by augmented tame approximation and specialization, together with its canonicity (since the tame isomorphism is in fact independent of infinite prime, cf. \cite{K-BE2}).

From the above considerations as well as from the main ideas of the proof of the Main Theorem, we therefore have the following

\begin{thm}[Main Theorem 2]\label{mainthmgen}
The groups $\Aut W_{n,\mathbb{Q}}$ and $\Aut P_{n,\mathbb{Q}}$ are canonically isomorphic.
\end{thm}

Furthermore, an elementary modification of the above line of reasoning leads to the proof of the general case of Conjecture \ref{mainconj} -- i.e. to the reformulation of Main Theorem 2 with $\mathbb{Q}$ replaced by an arbitrary field of characteristic zero.

\subsection*{Acknowledgments}

We thank I. Arzhantsev, R. Karasev, I. Karzhemanov, D. Kazhdan, E. Rips, G. Sharygin, E. Vinberg, I. Zhdanovskii and A. Zheglov for numerous helpful remarks.

This project is supported by the Russian Science Foundation grant No. 17-11-01377.

\section{Automorphism groups and the homomorphism $\Phi$}

We begin by making a few precise definitions of the varieties involved and then proceed to explain the construction of the morphism $\Phi$.

\subsection{Basic definitions}
The Weyl algebra $W_{n,\mathbb{C}}$ is a free $\mathbb{C}$-module; once the choice of generators is specified, we fix the basis of the $\mathbb{C}$-module consisting of monomials of the form
$$
x_1^{k_1}\ldots x_n^{k_n}d_1^{l_1}\ldots d_n^{l_n}\equiv x^K d^L
$$
($K$ and $L$ are multi-indices) thus imposing an ordering of generators that requires all $x$'s to be written ahead of all $d$'s.
\smallskip

\begin{Def} \label{defdeg}
We set every generator $x$ and $d$ to have degree one, and if $f$ is an element of $W_n$, its degree $\Deg f$ is defined as that of its highest-degree monomial (this definition does not depend on the ordering we fix due to the form of the commutation relations).
\end{Def}

\smallskip
\begin{Def} \label{defheight}
If $f$ is an element of $W_n$, its height is defined to be the degree of its smallest-degree monomial, in the ordering we have fixed:
$$
\Ht f = \inf \lbrace \Deg x^K d^L \;:\; x^K d^L\; \text{is in $f$ with non-zero coefficient}\rbrace.
$$
\end{Def}
The height is obviously sensitive to the ordering of generators we work with.
The notions of degree and height are identical for the polynomial algebra $P_n$ with the added benefit of indifference toward the ordering.

\smallskip
\begin{Def} \label{defset}
If
$$
(f_1,\ldots, f_m)
$$
is a finite set of polynomials (in $P_n$ or in $W_n$), its degree is defined to be the largest value of $\Deg f_k$, while its height is the smallest value of $\Ht f_k,$, when $k=1,\ldots, m$.
\end{Def}
\smallskip

Any $\mathbb{C}$-endomorphism $\varphi$ of $W_n$ (or $P_n$) is identified with the set of images
$$
(\varphi(x_1),\ldots, \varphi(x_n), \varphi(d_1),\ldots, \varphi(d_n))
$$
of the algebra generators. The degree and the height of $\varphi$ are then defined as above.
\smallskip

The group $\Aut W_{n,\mathbb{C}}$ admits a filtration by subsets
$$
\Aut^{\leq N} W_{n,\mathbb{C}} = \lbrace \varphi\in\Aut W_{n,\mathbb{C}}\;:\;\Deg \varphi\leq N\rbrace.
$$
An identical definition holds for $\Aut P_{n,\mathbb{C}}$.
\smallskip

The sets $\Aut^{\leq N} W_{n,\mathbb{C}}$ are in fact affine algebraic sets. Indeed, any element $\varphi$ of $\Aut^{\leq N} W_{n,\mathbb{C}}$ is identified with a set of $2n$ polynomials as above. These in turn are identified with an array of their coefficients, which together serve as coordinates of a point in an affine space of sufficiently large dimension. The requirement for $\varphi$ to be an automorphism imposes constraints on these coordinates which obviously have the form of polynomial equations. The same is true for $\Aut^{\leq N} P_{n,\mathbb{C}}$.

The sets $\Aut^{\leq N} W_{n,\mathbb{C}}$ are connected by means of the obvious embeddings
$$
\Aut^{\leq N} W_{n,\mathbb{C}} \rightarrow \Aut^{\leq N+1} W_{n,\mathbb{C}}
$$
which are Zariski-closed; the colimit of the inductive system of these mappings is of course the entire group $\Aut W_{n,\mathbb{C}}$. The same holds for $\Aut P_{n,\mathbb{C}}$.

\smallskip

Finally, the power series topology is defined as follows. Let us for definiteness consider $P_n$ (whose generators we briefly refer to as $z_1,\ldots z_{2n}$), and let
$$
I=(z_1,\ldots, z_{2n})
$$
be the ideal spanned by all its generators (we call it the \textbf{augmentation ideal}).
\begin{Def} \label{defaug}
For any positive integer $N$, define the subgroups of $\Aut P_{n,\mathbb{C}}$:
$$
H_N = \lbrace \varphi\in \Aut P_{n,\mathbb{C}}\;:\; \varphi(z_i)\equiv z_i\;(\text{mod}\;I^N)\rbrace.
$$
\end{Def}
The elements of $H_N$ are automorphisms which are identity modulo terms of height at least $N$. This specifies a proper system of neighborhoods of the neutral element of $\Aut P_{n,\mathbb{C}}$ and therefore defines a topology, which we refer to as the \textbf{augmentation topology} or, alternatively, \textbf{power series topology} (due to its being effectively the power series topology of $P_n$ induced by $I^N$). An analogous definition is made for $W_n$.

\begin{Def} \label{defrank}
The rank of an endomorphism, $\Rk(\varphi)$, is defined as the height
$$
\Ht(\varphi - \Id)
$$
where the difference between endomorphisms is an endomorphism obtained by taking the difference between the images of the generators.
\end{Def}
The rank measures how close the endomorphism is to the identity morphism. If $\varphi$ has rank $N$, then $\varphi$ is identity modulo $I^N$ as defined above.

In the sequel, we will work exclusively with automorphisms of height at least one: indeed, it suffices to prove the version of the Main Theorem for the subgroups of automorphisms with zero free term, as translations will then allow to reconstruct the isomorphism for every point. This is needed for the sake of approximation in the power series topology which we will discuss in the next subsection.
\subsection{Tame automorphisms}
Suppose first that $\mathbb{K}[x_1,\ldots, x_n]$ is the polynomial algebra over a field $\mathbb{K}$, and let $\varphi$ be an automorphism of this algebra.
\begin{Def} \label{defelement}
We call $\varphi$ an elementary automorphism, if it is of the form
\begin{equation*}
\varphi = (x_1,\ldots,\;x_{k-1},\;ax_k+f(x_1,\ldots,x_{k-1},\;x_{k+1},\;\ldots,\;x_n),\;x_{k+1},\;\ldots,\;x_n)
\end{equation*}
with $a\in\mathbb{K}^{\times}$.
\end{Def}
Observe that linear invertible changes of variables -- that is, transformations of the form
\begin{equation*}
(x_1,\;\ldots,\;x_n)\mapsto (x_1,\;\ldots,\;x_n)A,\;\;A\in\GL(n,\mathbb{K})
\end{equation*}
are realized as compositions of elementary automorphisms.
\smallskip
\begin{Def} \label{deftame}
A tame automorphism is, by definition, an element of the subgroup $\TAut \mathbb{K}[x_1,\ldots,x_n]$ generated by all elementary automorphisms. Automorphisms that are not tame are called wild.
\end{Def}

All automorphisms of $\mathbb{K}[x,y]$ are tame \cite{Jung, VdK}; whether all automorphisms of $\mathbb{K}[x_1,\ldots,x_n]$ (with $n$ even) are tame is unknown, but for $n=3$ the celebrated Nagata's automorphism is an example of an automorphism which is not tame \cite{Shes2, Shes3}.

\medskip

Similarly, the group of tame symplectomorphisms $\TAut P_{n,\mathbb{K}}$ is defined as the subgroup of those tame automorphisms of $\mathbb{K}[x_1,\ldots,x_n,p_1,\ldots,p_n]$ that preserve the Poisson bracket. If $\varphi$ is an elementary automorphism of $\mathbb{K}[x_1,\ldots,x_{n},p_1,\ldots,p_n]$, then in order to preserve the symplectic structure, it clearly must be either a linear symplectic change of generators:
\begin{equation*}
(x_1,\;\ldots,\;x_n,\;p_1,\;\ldots,\;p_n)\mapsto (x_1,\;\ldots,\;x_n,\;p_1,\;\ldots,\;p_n)A
\end{equation*}
with $A\in\Sp(2n,\mathbb{K})$ a symplectic matrix, or an elementary transformation of one of two following types:
\begin{equation*}
(x_1,\;\ldots,\;x_{k-1},\;x_k+f(p_1,\;\ldots,\;p_n),\;x_{k+1},\;\ldots,\;x_n,\;p_1,\;\ldots,\;p_n)
\end{equation*}
and
\begin{equation*}
(x_1,\;\ldots,\;x_{n},\;p_1,\;\ldots,\;p_{k-1},\;p_k+g(x_1,\;\ldots,\;x_n),\;p_{k+1},\;\ldots,\;p_n).
\end{equation*}

The subgroup of tame symplectomorphisms $\TAut P_{n,\mathbb{K}}$ is the group generated by such elementary symplectomorphisms.
\smallskip

The definition of the group $\TAut W_{n,\mathbb{K}}$ of tame automorphisms of the Weyl algebra mirrors that of tame symplectomorphisms, with the commuting generators $p_i$ replaced by $d_i$. We also note that in all cases, we do not include in our definition of automorphism those mappings that allow the images $\varphi(x)$ to contain a non-zero free part (an element in the span of the unit). This omittance is evidently not significant to our present discussion. In particular, the polynomials that define the elementary automorphisms as above have zero free term.

\medskip

The notion of tame automorphisms provides an excellent tool for approximating arbitrary polynomial automorphisms. Classical results in this regard were established by Anick \cite{An}. More specifically, let $\varphi\in\Aut \mathbb{K}[x_1,\ldots,x_n]$ be an automorphism of the polynomial algebra.
\begin{Def} \label{defappr}
We say that $\varphi$ is approximated by tame automorphisms if there is a sequence
\begin{equation*}
\psi_1,\;\psi_2,\ldots,\;\psi_k,\ldots
\end{equation*}
of tame automorphisms such that
\begin{equation*}
\Ht((\psi_k^{-1}\circ\varphi)(x_i)-x_i)\geq k
\end{equation*}
for $1\leq i\leq n$ and all $k$ sufficiently large. In other words, an automorphism is approximated by a sequence of tame automorphisms if there is such a sequence that converges to this automorphism in power series topology.
\end{Def}
Observe that any tame automorphism $\psi$ is approximated by itself -- that is, by a stationary sequence $\psi_k=\psi$.

The following theorem is true.

\begin{thm} \label{thm-anick}
Let $\varphi=(\varphi(x_1),\;\ldots,\;\varphi(x_n))$ be an automorphism of the polynomial algebra $\mathbb{K}[x_1,\ldots,x_n]$ over a field $\mathbb{K}$ of characteristic zero, such that its Jacobian
\begin{equation*}
\J(\varphi)=\Det \left[\frac{\partial \varphi(x_i)}{\partial x_j}\right]
\end{equation*}
is equal to $1$. Then there exists a sequence $\lbrace \psi_k\rbrace\subset \TAut \mathbb{K}[x_1,\ldots,x_n]$ of tame automorphisms approximating $\varphi$.
\end{thm}

As it turns out, this theorem of Anick has an analogue in the symplectic setting, which we have proved recently. The theorem is a natural idea in the construction of the inverse homomorphism in Conjecture \ref{mainconj}, however approximation in the augmented setting is required to properly define the lifting.

\begin{thm}\label{thmkge}
Let $\sigma=(\sigma(x_1),\;\ldots,\;\sigma(x_n),\;\sigma(p_1),\;\ldots,\;\sigma(p_n))$ be a symplectomorphism of $\mathbb{K}[x_1,\ldots,x_n,p_1,\ldots,p_n]$ with unit Jacobian.
Then there exists a sequence $\lbrace \tau_k\rbrace\subset \TAut P_n(\mathbb{K})$ of tame symplectomorphisms approximating $\sigma$.
\end{thm}
The proof is detailed in \cite{KGE}. The theorem states that the subgroup of tame symplectomorphisms is dense in augmentation topology in the group of all polynomial symplectomorphisms.

\subsection{Augmented and skew Weyl and Poisson algebras}

In the proof of our main result we make use of the augmented (deformed, or quantized) versions of $W_n$ and $P_n$, which we now define. In order to deform the Weyl algebra $W_n$, we introduce the augmentation parameter $h$ and modify the commutator between $d$ and $x$ by setting
$$
[d_i, x_j] = h \delta_{ij}.
$$
Alternatively, one can start with the free algebra $\mathbb{K}\langle a_1,\ldots, a_n, b_1,\ldots, b_n, c\rangle$ and take the quotient with respect to the following set of identities:
$$
a_ia_j-a_ja_i,\;\;b_ib_j-b_jb_i\;\;,
b_ia_j-a_jb_i-\delta_{ij}c,\;\;
a_ic-ca_i,\;\;
b_ic-cb_i.
$$
The quotient algebra, which we denote by $W^h_n$ (or $W^h_{n,\mathbb{K}}$ to indicate the base field) is the augmented Weyl algebra.

Similarly, we may distort the Poisson bracket of $P_n$:
\begin{equation*}
\lbrace p_i,x_j\rbrace = h\delta_{ij}
\end{equation*}
to reflect the augmentation of $W_n$ into $W^h_n$ in the classical counterpart. (Here we have renamed the generators $z_i$ into $x_i$ and $p_i$, according to their behavior with respect to the Poisson bracket. The notation is standard.) The resulting polynomial algebra will be denoted by $P^h_n$.
\medskip

The algebras $W^h_n$ and $P^h_n$ are connected by an analogue of the Kontsevich homomorphism (which is constructed for the non-augmented case below).

\smallskip

In order to create a situation in which the power series topology is well respected by the morphisms, we distort the augmentation further by introducing a pair of \textbf{skew augmented} algebras $W_{n,\mathbb{C}}^h[k_{ij}]$ and $P_{n,\mathbb{C}}^h[k_{ij}]$ (which correspond to the $h$-augmented Weyl and Poisson algebras, respectively).

These are defined as follows. Let the augmented Poisson generators be denoted by $\xi_i$ with $1\leq i\leq 2n$ and let $[k_{ij}]$ be an antisymmetric matrix of central variables. The algebra $P_{n,\mathbb{C}}^h[k_{ij}]$ is generated by $2n$ commuting variables $\xi_i$, the augmentation variable $h$ and the variables $[k_{ij}]$ (thus being the polynomial algebra in these variables); the Poisson bracket is defined on the generators $\xi_i$:
$$
\lbrace \xi_i,\xi_j\rbrace = hk_{ij}.
$$
The bracket of any element with $h$ or with any of the $k_{ij}$ is zero. The skew version of the algebra $W_{n,\mathbb{C}}$ is defined analogously.

It is for these auxiliary algebras that the analogue of the Kontsevich conjecture can be proved. Once that is done, the proof of the main results is finalized by the specialization argument. As we shall see, this final step will require the extension of the lifting map to points rational in $h$; we note, however, that tame approximation in the $h$-augmented and skew augmented cases is polynomial in $h$.

\smallskip

From the standpoint of our context, much of the theory of the Weyl and Poisson algebras remains unchanged by the augmentation as well as by skew augmentation: the construction of the homomorphism $\Phi$ is identical, the proof of the fact that $\Phi$ is a morphism of the normalized varieties is essentially the same, and the tame augmented symplectomorphism approximation for the $h$-augmented algebra is established in a way almost identical to the one presented in \cite{KGE} -- the $h$-augmented version of Theorem \ref{thmkge} is valid. What does change is the behavior of $\Phi$  with respect to the approximation in the skew case, as the deformed version of $\Phi$ will be continuous in the power series topology. This is the main point of the augmentation.

\subsection{Approximation and the singularity trick}

The proof of Propositions \ref{skewthetaprop} and utilizes a certain "singularity trick". Essentially it is a technique that, by examining certain curves in $\Aut$, allows one to efficiently control the height of the higher-degree terms in an automorphism and its image under $\Phi$ which is near the identity automorphism. In the main text, the two statements that comprise this technique are Lemma \ref{lem1} and Proposition \ref{singtrick}.

\smallskip

The idea of Poisson (and Weyl) structure augmentation, which enables the proof of the Main Theorem, and the singularity trick complement each other and {\bf constitute the cornerstone of our approach to the polynomial symplectomorphism quantization (lifting) problem} by way of enabling a stronger form of tame approximation. This advantage is offset by the need to specialize the deformation parameters (Planck constant) in order to return to the non-augmented case, as well as by the necessity to introduce the general form of commutation relations (which we call \textbf{skew} augmented algebra structure, owing to its antisymmetry) however the specialization can be proven correct, as in the last part of the proof of the Main Theorem.

\smallskip

The singularity trick is rather useful when dealing with direct systems of varieties equipped with the power series topology. It was first introduced in our prior paper \cite{KBYu} (Theorem 3.2, Lemma 3.5, 3.6 and 3.7). In our proof, this technique is employed in an essentially the same manner, and is contained in Lemma \ref{lem1} and Proposition \ref{singtrick}.

\smallskip

For the archetypal case of the automorphism group of the commutative polynomial algebra $\mathbb{K}[x_1,\ldots, x_n]$, the situation is as follows.

Let $L=L(t)$ be a curve of linear automorphisms, i.e. a curve
$$L \subset \Aut(K[x_1,\dots,x_n]),$$ whose points are linear substitutions. Suppose that, as $t$ tends to zero, the $i$-th eigenvalue of the matrix $L(t)$ (corresponding to the linear changes of variables) also tends to zero as $t^{k_i}$, $k_i\in\mathbb{N}$. Such a family always exists.

Suppose now that the degrees $\lbrace k_i,\;i=1,\ldots n\rbrace$ of singularity of eigenvalues at zero are such that for every pair $(i,j)$, if $k_i\neq k_j$, then there exists a positive integer $m$ such that
$$
\text{either\;\;} k_im\leq k_j\;\;\text{or\;\;}k_jm\leq k_i.
$$

The largest such $m$ we will call the \textbf{order} of $L(t)$ at $t=0$. As $k_i$ are all set to be positive integer, the order equals the integer part of $\frac{k_{\text{max}}}{k_{\text{min}}}$.

Let $M\in \Aut_0(K[x_1,\dots,x_n])$ be a polynomial automorphism.
\begin{lem}    \label{Lm2} The curve $L(t)ML(t)^{-1}$ has no singularity at zero for any $L(t)$ of order $\leq N$ if and only if $M\in \hat{H}_N$, where $\hat{H}_N$ is the subgroup of automorphisms which are homothety modulo the $N$-th power of the
augmentation ideal $(x_1,\ldots, x_n)$.
\end{lem}

The proof can be found in \cite{KBYu}. Also, a direct analogue of this lemma is Proposition \ref{singtrick}, with proof borrowing its essential features from \cite{KBYu}.

\begin{remark}\label{weakmotivation}
The emergence of the singularity trick (together with the general consideration of singularities of $\Ind$-schemes) may be viewed as a reflection of the infinite-dimensional nature of the problem. Indeed, in the context of finite-dimensional algebraic groups, the most natural approach to a problem such as Conjecture \ref{mainconj} would be the construction of a morphism which induces an isomorphism of the Lie algebras. However, as was pointed out in \cite{BKK1} (specifically in Section 3 and also in Remark 2 of \cite{BKK1}), the naive infinite-dimensional translation of this approach is unsatisfactory. In fact, the Lie algebras (defined as the algebras of derivations) of $W_{n,\mathbb{Q}}$ and $P_{n,\mathbb{Q}}$ are not isomorphic to each other. Therefore, the positivity of Conjecture \ref{mainconj} leads to the breakdown of the Lie algebra isomorphism approach typical of finite-dimensional cases. This pathological infinite-dimensional effect could be the result of the $\Ind$-schemes being singular at every point \cite{BKK1}.

\end{remark}

We now turn to the definition of the candidate homomorphism for Conjecture \ref{mainconj}.

\subsection{Homomorphism $\Phi$}

As mentioned in the introduction, the homomorphism
\begin{equation*}
\Phi: \Aut W_{n,\mathbb{C}}\rightarrow \Aut P_{n,\mathbb{C}}
\end{equation*}
is constructed by means of a certain ultraproduct decomposition. The few items needed are hereby presented. The detailed discussion and proofs can be found in \cite{Tsu2} or in \cite{K-BE2}.

\medskip

Let $\mathcal{U}\subset 2^{\mathbb{N}}$ be a fixed non-principal ultrafilter on the set of positive integers (which is the index set for the sequences and the product defined below). The ultrafilter $\mathcal{U}$ induces an equivalence relation on the set of all countable sequences of prime numbers: if
$$
(p_m) = \lbrace p_1,\ldots, p_m,\ldots\rbrace\;\;\text{and}\;\;(q_m)
$$
are two sequences, they are equivalent modulo $\mathcal{U}$ if the set
$$
\lbrace m\;:\;p_m=q_m\rbrace
$$
is in $\mathcal{U}$. Since $\mathcal{U}$ is non-principal, the equivalence $\sim_{\mathcal{U}}$ partitions the set of all prime number sequences into equivalence classes which will be of two types: classes that contain a (necessarily unique) stationary sequence and classes whose every representative is unbounded.\footnote{This statement is an easy consequence of the ultrafilter properties.} The classes $[p]$ of the second kind are referred to as infinite primes. The name is in agreement with the notion of prime element in the ring $^*\mathbb{Z}$ of hyperintegers, which can be obtained as the quotient of the direct product of a countable set of copies of $\mathbb{Z}$ modulo the minimal prime ideal
$$
(\mathcal{U}) = \lbrace (a_m)\in \prod_{m\in\mathbb{N}} \mathbb{Z}\;:\;\text{the set}\;\lbrace m\;:\;a_m=0\rbrace\;\text{is in}\;\mathcal{U}\rbrace.
$$
The fact that $(\mathcal{U})$ is a minimal prime ideal is true as long the components in the direct product are integral domains. If all the components are fields, then $(\mathcal{U})$ is also maximal, since the product of fields is always von Neumann regular. Note also that taking the quotient by $(\mathcal{U})$ is the same as partitioning into classes modulo $\mathcal{U}$ in the sense of the equivalence relation defined above.

\smallskip

Let $(p_m)$ be a prime number sequence which defines an infinite prime $[p]$ with respect to the fixed ultrafilter $\mathcal{U}$. Consider, for each $p_m$ in the sequence, the algebraically closed field $\mathbb{F}_{p_m}$ of characteristic $p_m$. The following lemma is true.

\begin{lem} \label{redlemma1}
The ultraproduct
$$
\left(\prod_{m\in\mathbb{N}}\mathbb{F}_{p_m}\right)\;/\;\mathcal{U}
$$
has cardinality of the continuum and is an algebraically closed field of characteristic zero.
\end{lem}
The algebraic closedness and characteristic zero are straightforward. The cardinality is a little less straightforward, the argument may be found in \cite{Tsu2} or in \cite{K-BE2}.
As a corollary of the well-known Steinitz's theorem, we then have
\begin{lem} \label{redlemma2}
$$
\left(\prod_{m\in\mathbb{N}}\mathbb{F}_{p_m}\right)\;/\;\mathcal{U}\simeq \mathbb{C}.
$$
\end{lem}

This lemma constitutes the reduction of the ground field modulo infinite prime: the scalars are represented as modulo $\mathcal{U}$ classes of sequences $(a_m)$, with (most of) the elements $a_m$ being algebraic over $\mathbb{Z}_{p_m}$.

\smallskip

The next item concerns the properties of the Weyl algebra $W_{n,\mathbb{K}}$ in positive characteristic. We have already mentioned that if $\Char\mathbb{K}=p>0$, then the center
$$
C(W_{n,\mathbb{K}})\simeq \mathbb{K}[z_1,\ldots, z_{2n}],
$$
where $z_i$ are $p$-th powers of the generators of $W_n$ (the proof is an easy exercise). Next, the following important property holds.
\begin{lem} \label{redlemma3}
If $\varphi$ is an endomorphism of $W_{n,\mathbb{K}}$ (with $\Char\mathbb{K}=p$), then
$$
\varphi (C(W_{n,\mathbb{K}}))\subset C(W_{n,\mathbb{K}})
$$
-- that is, the endomorphism $\varphi$ induces an endomorphism of the center.
\end{lem}
An elegant proof can be found in \cite{Tsu1} (Lemma 4).

\smallskip

Next we define the Poisson bracket on the center of $W_{n,\mathbb{K}}$. Set for definiteness $\mathbb{K}=\mathbb{F}_{p}$. Then the center is given by
$$
\mathbb{F}_{p}[x_1^p,\ldots, x_n^p,d_1^p,\ldots, d_n^p]
$$
($x_i$ and $d_j$ are Weyl algebra generators) and therefore contains as a subalgebra the algebra
$$
\mathbb{Z}_{p}[x_1^p,\ldots, x_n^p,d_1^p,\ldots, d_n^p]
$$
of polynomials over $\mathbb{Z}_{p}$. Now, for any two elements $a,b\in \mathbb{Z}_{p}[x_1^p,\ldots, x_n^p,d_1^p,\ldots, d_n^p]$, define
\begin{equation*}
\lbrace a,b\rbrace=-\rho\left(\frac{[a_0,b_0]}{p}\right).
\end{equation*}
Here
$$
\rho:W_{n,\mathbb{Z}}\rightarrow W_{n,\mathbb{Z}_p}
$$
is the (usual) modulo $p$ reduction of the Weyl algebra over $\mathbb{Z}$, and $a_0$ and $b_0$ are elements of $W_{n,\mathbb{Z}}$ in the pre-images $\rho^{-1}(a_0)$ and $\rho^{-1}(b_0)$, respectively. It can be checked that the bracket is well defined, and the bracket admits a straightforward extension to the entire center. It is also bilinear, takes values in the center, and satisfies the required Leibnitz and Jacobi identities. Finally, it is standard in the sense that
\begin{equation*}
\lbrace d_i^p,x_j^p\rbrace=\delta_{ij}.
\end{equation*}

As the Poisson bracket on the center is induced by the Weyl algebra commutator, the following lemma holds.
\begin{lem} \label{redlemma4}
If $\varphi$ is an endomorphism of $W_{n,\mathbb{F}_p}$, then the induced endomorphism $\varphi^c$ of the center $C(W_{n,\mathbb{F}_p})$ preserves the Poisson bracket.
\end{lem}

\medskip

The last item is the ultraproduct of Weyl algebras. Just as with fields, the ultrafilter $\mathcal{U}$ induces a congruence on the direct product of Weyl algebras, so that one may define the algebra
\begin{equation*}
\mathcal{A}_n(\mathcal{U},[p])=\left(\prod_{m\in\mathbb{N}}A_{n,\mathbb{F}_{p_m}}\right)/\mathcal{U}.
\end{equation*}
As the coefficients of elements of $\mathcal{A}_n(\mathcal{U},[p])$ take values in $\left(\prod_{m\in\mathbb{N}}\mathbb{F}_{p_m}\right)\;/\;\mathcal{U}$, the algebra $\mathcal{A}_n(\mathcal{U},[p])$ is a $\mathbb{C}$-algebra, which contains the Weyl algebra $W_{n,\mathbb{C}}$ as a proper subalgebra (of sums of monomials of globally bounded degree). However, unlike $W_{n,\mathbb{C}}$, the ultraproduct algebra has a huge center given, obviously, by the ultraproduct of centers $C(W_{n,\mathbb{F}_{p_m}})$. Therefore:
\begin{lem} \label{redlemma5}
There is an injection
$$
\mathbb{C}[z_1,\ldots,z_{2n}]\rightarrow C(\mathcal{A}_n(\mathcal{U},[p])).
$$
\end{lem}
\smallskip

The homomorphism $\Phi$ is constructed in the following way. Given an endomorphism $\varphi$ of $W_{n,\mathbb{C}}$, one represents its coefficients (that is, the coefficients of the images $\varphi(x_i)$ and $\varphi(d_j)$ in the standard $\mathbb{C}$-basis) as elements of the reduced direct product $\left(\prod_{m\in\mathbb{N}}\mathbb{F}_{p_m}\right)\;/\;\mathcal{U}$ and constructs from them an array of endomorphisms $\varphi_{p_m}$ of $W_{n,\mathbb{F}_{p_m}}$ (as there are only finitely many constraints on the coefficients of $\varphi$, this reconstruction will be possible for almost all -- in the sense of $\mathcal{U}$ -- indices $m$). Then one restricts to the center to obtain $\varphi_{p_m}^c$. Since these will be endomorphisms of bounded degree, they will give an endomorphism
$$
C(\mathcal{A}_n(\mathcal{U},[p]))\rightarrow C(\mathcal{A}_n(\mathcal{U},[p]))
$$
of the center of the ultraproduct algebra which will map its subalgebra $\mathbb{C}[z_1,\ldots,z_{2n}]$ to itself.
\smallskip

One could stop here and denote the resulting symplectomorphism of $\mathbb{C}[z_1,\ldots,z_{2n}]$ by $\Phi(\varphi)$; we will, however, slightly twist the morphism in order to get rid of the explicit dependence of the choice of infinite prime $[p]$. Note that in characteristic $p$ the correpsondence $\varphi\mapsto \varphi^c$, when $\varphi$ is a linear change of variables (given by a symplectic matrix $A=(a_{ij})$), produces a symplectomorphism whose matrix is $A^c=(a_{ij}^p)$. If one applied the inverse Frobenius automorphism to the scalars (takes the $p$-th root), one would recover a symplectomorphism given by exactly the same set of coordinates (entries of the matrix in this case) as that of the Weyl algebra morphism. Now, any automorphism of the base field induces an automorphism of the algebra, so this added procedure does not ruin the homomorphic property of the correspondence.
\smallskip

Combined in the ultraproduct, the (inverse) Frobenius automorphisms give rise to an automorphism of $\mathbb{C}$, which again induces an automorphism of the algebras. The entire correspondence now consists of the following steps: start with a Weyl algebra automorphism, take it to the ultraproduct, restrict to the center by $\varphi\mapsto \varphi^c$ (for almost all $p_m$), twist by inverse Frobenius, reassemble as a $\mathbb{C}$-symplectomorphisms. Denote the resulting homomorphism by $\Phi(\varphi)$. By Lemma \ref{redlemma4}, $\Phi(\varphi)$ preserves the Poisson bracket -- in other words, it is a symplectic map. To complete the construction, we state the following lemma, which is proved as part of the main proposition of \cite{Tsu2}.
\begin{lem} \label{redlemma6}
The resulting polynomial endomorphism $\Phi(\varphi)$ is an automorphism if and only if $\varphi$ is an automorphism.
\end{lem}
This concludes the definition of the homomorphism $\Phi$ which is our candidate for Conjecture \ref{mainconj}.

\smallskip

Some of the properties of $\Phi$ may be established right away, without the need for a complicated development such as approximation and lifting. For the proof of the next proposition, cf. \cite{BKK1}.
\begin{prop} \label{phiismono}
$\Phi$ is injective.
\end{prop}
Surjectivity of $\Phi$ would imply Conjecture \ref{mainconj}; there seems to be no direct way to obtain this result (in a manner similar to some of the proofs of \cite{Tsu1}, for instance). Constructing an inverse homomorphism via approximation and lifting of symplectomorphisms was viewed as a more viable approach.

\smallskip

It is interesting to note another conjecture in connection to the morphism $\Phi$, or rather to the way it is constructed.
\begin{conj} \label{secondconj}
The homomorphism $\Phi$ is independent of the choice of infinite prime $[p]$.
\end{conj}
The conjecture is not vacuous, for while we did get rid of the explicit dependence of the coefficients on $[p]$ by inverse Frobenius twist, the construction of $\Phi$ via a reduced direct product decomposition allows for an implicit dependence in the following way. Observe that at the stage of decomposed automorphism $\varphi$ (with induced positive characteristic automorphisms $\varphi_{p_m}$ for almost all $p_m$), the place $p_m$ (i.e. the value $p_m$ at index $m$) could have an influence on which monomials are not zero in the images $\varphi_{p_m}(x_i)$, and while the highest-degree monomials do not change (which follows easily if one recalls what $\varphi\mapsto\varphi^c$ is for every $p$ and keeps in mind the commutation relations), it cannot be readily obtained that the rest of the monomials are fixed. Any direct approach to nail down these lower-degree places seems to be insufficient, therefore an indirect method of handling the places, by means of discovering certain rigid properties of $\Phi$, has to be implemented. This is what has been done by us recently, and Conjecture \ref{secondconj} is indeed positive. The conjecture is the subject of our work \cite{K-BE2}, and the proof (specifically its later stage) ultimately uses the properties of the same augmented versions of $W_n$ and $P_n$ we present in this paper.

\medskip

We finish this subsection by formulating the crucial theorem which allows $\Phi$ to be used in the subsequent development of lifting. It first appeared and was proved in \cite{BKK1}.
\begin{thm} \label{idtame}
The homomorphism $\Phi$ induces an isomorphism of the tame subgroups:
\begin{equation*}
\Phi:\TAut W_{n,\mathbb{C}}\xrightarrow{\sim} \TAut P_{n,\mathbb{C}}.
\end{equation*}
\end{thm}
The theorem makes lifting of tame symplectomorphisms possible at once, and the approximation as developed in \cite{KGE} allows lifting of arbitrary symplectomorphisms to automorphisms of power series completion of $W_n$, so that a local version of Conjecture \ref{mainconj} holds.

To conclude this subsection, we state the following theorem.
\begin{thm} \label{thmgabber}
The mappings
$$
\Phi_N:\Aut^{\leq N} W_{n,\mathbb{C}}\rightarrow \Aut^{\leq N} P_{n,\mathbb{C}}
$$
induced by $\Phi$
are morphisms of algebraic varieties.
\end{thm}

The proof can be found in \cite{K-BE2}. This theorem has an exact (and crucial to our approach) analogue in the setting of the quantized algebras $W_n$ and $P_n$, which we state for the reference in the next section.



\section{Augmented Weyl algebra structure}

In order to resolve the symplectomorphism lifting problem and construct the inverse to the homomorphism $\Phi$, we introduce the augmented and skew augmented Weyl and Poisson algebras.

The \emph{augmented}, or $h$-\emph{augmented} Weyl algebra $W^h_{n, \mathbb{C}}$ is defined as the quotient of the free algebra on $(2n+1)$ indeterminates $\mathbb{C}\langle a_1,\ldots, a_n, b_1,\ldots, b_n, c\rangle$ by the two-sided ideal generated by elements
$$
a_ia_j-a_ja_i,\;\;b_ib_j-b_jb_i\;\;,
b_ia_j-a_jb_i-\delta_{ij}c,\;\;
a_ic-ca_i,\;\;
b_ic-cb_i.
$$
The algebra $W^h_{n, \mathbb{C}}$, in other words, differs from $W_{n, \mathbb{C}}$ in the form of the commutation relations -- in the case of $W^h_{n, \mathbb{C}}$, the coordinate-momenta pairs of generators commute into $h$ (which is added as a central variable to the algebra; the variable $h$ thus somewhat resembles the Planck constant) -- and one can return to the non-augmented algebra $W_{n, \mathbb{C}}$ by \emph{specializing} the augmentation parameter to $h=1$. The augmented Poisson algebra, denoted by $P^h_{n, \mathbb{C}}$, is defined similarly: one adds the variable $h$ to the commutative polynomial algebra of $2n$ generators and endows it with the Poisson bracket defined as:
$$
\lbrace p_i, x_j\rbrace = h\delta_{ij}.
$$

\smallskip

It can be verified that these new algebras behave in a way almost identical to the one we described in the prequel; in particular, the notions of tame automorphism, tame (modified) symplectomorphism and homomorphism
\begin{equation*}
\Phi:\Aut^{\leq N}(W^h_{n,\mathbb{C}})\rightarrow \Aut^{\leq N}(P^h_{n,\mathbb{C}}).
\end{equation*}
(defined for a fixed infinite prime) which is identical on the tame points, are present. We also note that the action of any $h$-augmented automorphism (or, correspondingly, symplectomorphism) on $h$ is necessarily a dilation
$$
h\mapsto \lambda h
$$
where $\lambda$ is a constant. Indeed, the image of $h$ cannot contain monomials proportional to $x_i$ or $p_j$ (otherwise the commutation relations will not hold), and it cannot be a polynomial in $h$ of degree greater than one, as in that case
Also, the proof of the counterpart of the Theorem \ref{thmgabber} is established in a similar fashion.
\begin{thm} \label{thmgabberh}
The mappings
$$
\Phi^h_N:\Aut^{\leq N} W^h_{n,\mathbb{C}}\rightarrow \Aut^{\leq N} P^h_{n,\mathbb{C}}
$$
induced by $\Phi^h$
are morphisms of normalized algebraic varieties.
\end{thm}

\smallskip

As we shall see, one can prove the counterpart to the Conjecture \ref{mainconj} for these augmented algebras, and then demonstrate that the specialization to $h=1$ yields the isomorphism between automorphism groups of the non-augmented algebras. The construction of the augmented version of the isomorphism, however, requires to further modify the algebras by making the commutators between $d_i$ and $x_j$ nonzero for $i\neq j$.

This pair of auxiliary, \emph{skew augmented} algebras, denoted by $W_{n,\mathbb{C}}^h[k_{ij}]$ and $P_{n,\mathbb{C}}^h[k_{ij}]$ (which correspond to augmented Weyl and Poisson algebras, respectively), are defined as follows. Let the augmented Poisson generators be denoted by $\xi_i$ with $1\leq i\leq 2n$, which we will call the main generators, (the passage from $x_i$ and $p_j$ to $\xi_i$ is made for the sake of uniformity of notation -- and in fact, from the viewpoint of the singularity trick which we use below in order to establish canonicity of the lifting, all of these generators are on equal footing, unlike the standard form $x_i$ and $p_j$, for which only symplectic transformations are permitted), and let $[k_{ij}]$ be a skew-symmetric array (a skew matrix) of central variables. The algebra $P_{n,\mathbb{C}}^h[k_{ij}]$ is generated by $2n$ commuting variables $\xi_i$, the augmentation variable $h$ and the variables $[k_{ij}]$ (thus being the polynomial algebra in these variables); the Poisson bracket is defined on the generators $\xi_i$:
$$
\lbrace \xi_i,\xi_j\rbrace = hk_{ij}.
$$
The bracket of any element with $h$ or with any of the $k_{ij}$ is zero.

The skew version of the algebra $W_{n,\mathbb{C}}$ is defined analogously.

\smallskip

It is easily seen that the new algebras essentially share the positive-characteristic properties with $W_n$ and $P_n$, from which it follows that a mapping
$$
\Phi^{hk}:\Aut W_{n,\mathbb{C}}^h[k_{ij}]\rightarrow \Aut P_{n,\mathbb{C}}^h[k_{ij}]
$$
analogous to $\Phi$ and $\Phi^h$ can be defined for every infinite prime $[p]$. In a manner identical to the previous section it can be established that this mapping consists of a system of morphisms of the normalized varieties $\Aut^{\leq N}$, thus yielding the skew augmented analogue of Theorems \ref{thmgabber} and \ref{thmgabberh}.

\begin{thm}\label{thmgabberskew}
The mappings $\Phi^{hk}_N$ are morphisms of normalized varieties.
\end{thm}

\smallskip


The plan of the proof of the main theorem goes as follows. Given that in all three considered cases -- the non-augmented, the $h$-augmented and the skew augmented case -- the $\Ind$-morphism between (the normalizations of) $\Ind$-varieties of automorphisms is well defined, we will examine its properties. In particular, we are going to establish the continuity of the morphism $\Phi^{hk}$ -- or, to be more precise, its restriction to a certain subspace -- in the power series topology (defined by the choice of grading below). That result, together with the tame approximation and $\Phi^{hk}$ being the identity map on the tame automorphisms -- a property which also holds in all three considered cases -- will allow us to prove that the lifted limits of tame sequences are independent of the choice of the converging sequence (canonicity of lifting) and then demonstrate that the lifted limits are given by polynomials and not power series, i.e. that the lifted limits are (skew augmented Weyl algebra) automorphisms. Effectively we will establish the skew augmented version of the Kontsevich conjecture, or rather the more relevant isomorphism between subgroups $\Aut_{k} $ of automorphisms which act linearly on the auxiliary variables $k_{ij}$. Most of the conceptually non-trivial topological machinery -- namely, the singularity trick mentioned in the introduction, are employed at this first stage. In fact, the good behavior of the skew augmented algebras with respect to the singularity trick is the sole reason for introducing these algebras in our proof.

Next, we will connect the skew augmented algebras with the $h$-augmented algebras by means of a localization argument. Once this is done, the establishing of canonicity of lifting and polynomial nature of the lifted limits in the $h$-augmented case becomes a fairly straightforward affair.

Lastly, in order to demonstrate that the results for the $h$-augmented algebras imply the Kontsevich isomorphism, we will need to specialize to $h=1$. The procedure requires some effort, and in fact extension of the domain for the constructed inverse morphism will be needed. The procedure will finalize the proof.

\subsection{Continuity of $\Phi^{hk}$ and the singularity trick}

We will now study the power series topology induced on subgroups of skew augmented algebra automorphisms which are linear on $k_{ij}$. As usual, the topology is metric, and it is induced by the grading specified according to the following assignment of degrees to the generators:

$$\Deg h = 0,$$ $$\Deg k_{ij} = 2,$$ $$\Deg \xi_i = 1.$$ Note that since $h$ and $k_{ij}$ appear as products in the commutation relations, one could assign degree two to the augmentation parameter $h$ and degree zero to the skew-form variables $k_{ij}$, in analogy with the case of augmented algebra $P^h_n$, while essentially preserving the $\Ind$-scheme structure of $\Aut$.

\smallskip

The metric which induces the power series topology is defined as
$$
\rho(\varphi, \psi) = \exp(-\Ht(\varphi-\psi))
$$
where
$$
\varphi-\psi = (\varphi(\xi_1)-\psi(\xi_1),\ldots,\varphi(\xi_{2n})-\psi(\xi_{2n}),\ldots)
$$
is the algebra endomorphism defined by its images (on $\xi_i$, $h$ and $k_{ij}$), and the \textbf{height} $\Ht(\varphi)$ of an endomorphism is defined as the minimal total degree $m$ such that in one of the generator images under $\varphi$ a non-zero homogeneous component of degree $m$ exists.

Symbolically, we say that the power series topology is defined via the powers of the augmentation ideal $I$
$$
I=(\xi_1,\ldots, \xi_{2n}, h, \lbrace k_{ij}\rbrace)
$$
just as it is so in the commutative case, when every variable carries degree one.

The system of neighborhoods $\lbrace H_N\rbrace$ of the identity automorphism in $\Aut P_{n,\mathbb{C}}^h[k_{ij}]$ is defined by setting
$$
H_N = \lbrace g\in \Aut P_{n,\mathbb{C}}^h[k_{ij}]\;:\; g(\eta) \equiv \eta \;(\text{mod}\;I^N)\rbrace
$$
(here $\eta$ denotes any generator in the set $\lbrace \xi_1,\ldots, \xi_{2n},h, \lbrace k_{ij}\rbrace\rbrace$, so that elements of $H_N$ are precisely those automorphisms which are identity modulo terms which lie in $I^N$; again, the phrase "mod $I^N$" is short-hand for the distance as defined above).

Similar notions of grading, topology, and system of standard neighborhoods of a point, are valid for the algebra $W_{n,\mathbb{C}}^h[k_{ij}]$ (once the proper ordering of the generators in the chosen set is fixed). The neighborhoods of the identity point for this algebra will be denoted by $G_N$.

The point of introducing the skew algebras $W_{n,\mathbb{C}}^h[k_{ij}]$ and $P_{n,\mathbb{C}}^h[k_{ij}]$ is that a certain singularity analysis procedure (the singularity tricks mentioned in the introduction) can be implemented for these algebras in full analogy with the case of the commutative polynomial algebra processed in our preceding study \cite{KBYu}, while on the other hand there seems to be no straightforward way to execute the singularity trick for the algebras $W^h_n$ and $P^h_n$. Furthermore, after adjunction of $k_{ij}^{-1}$ (together with the entries of the inverse matrix) and extension of scalars (the localization procedure) one can embed the $h$-augmented $\mathbb{C}$-algebras $W^h_n$ and $P^h_n$ in the skew augmented algebras over the larger coefficient ring, thus connecting the $h$-augmented automorphisms with the skew augmented ones which are linear on $k_{ij}$, as we shall see below.

\smallskip

We will establish the continuity of the direct morphism $\Phi^{hk}$ and perform the singularity trick in the following stable form.

Consider the algebra $P_{n+1,\mathbb{C}}^h[k_{ij}]$ with $(2n+2)$ main generators $\lbrace \xi_1,\ldots, \xi_{2n}, u,v\rbrace$. Let
$$
\Aut_{u,v,k} P_{n+1,\mathbb{C}}^h[k_{ij}]
$$
denote the set of all automorphisms $\varphi$ of $P_{n+1,\mathbb{C}}^h[k_{ij}]$ such that:

1. $\varphi(\xi_i) = \xi_i + S_i$, where $S_i$ is a polynomial (in $\xi_i$, $u$, $v$, $h$ and $k_{ij}$) such that its height with respect to $\lbrace \xi_1,\ldots, \xi_{2n}, u,v\rbrace$ is at least two.

2. $\varphi(u) = u$, $\varphi(v) = v$.

3. $\varphi(k_{ij})$ is a $\mathbb{C}$-linear combination of $k_{ij}$, i.e. $\varphi\in\Aut_{k}P_{n+1,\mathbb{C}}^h[k_{ij}]$.

Define the grading as before: $\xi_i$, $u$, $v$ carry degree one, $h$ carries degree zero, and $k_{ij}$ carry degree two.

Denote by $H_N^{u,v,k}$ the subgroups of $\Aut_{u,v,k} P_{n+1,\mathbb{C}}^h[k_{ij}]$ consisting of elements which are the identity map modulo terms of height $N$ with respect to the grading defined above. Also, the definition is repeated for the skew augmented Weyl algebra $W_{n+1,\mathbb{C}}^h[k_{ij}]$; the resulting subgroups are denoted by $G_N^{u,v,k}$.

\smallskip

The purpose of the singularity trick set up below is the proof of the following result, which establishes continuity of the direct morphism.

\begin{prop}\label{skewthetaprop}
If $\Phi^{hk}$ is the restriction of the direct morphism to $\Aut_k$, then
$$
\Phi^{hk}(G_N^{u,v,k}) \subseteq H_N^{u,v,k}
$$
for every $N$.
\end{prop}

The singularity trick is essentially a criterion for an automorphism $\varphi$ to be an element of $H_N^{u,v,k}$, expressed in terms of \emph{asymptotic behavior of certain parametric families} associated to it. The parametric families of automorphisms are constructed from $\varphi$ by conjugating it with $\mathbb{C}$-linear changes of the main generators (the latter are given by the set $\lbrace \xi_1,\ldots, \xi_{2n}\rbrace$). Such parameterized variable changes are given by $(2n+2)$ by $(2n+2)$ matrices $\Lambda(t)$ with
$$
(\xi_1,\ldots, \xi_{2n},u,v)\mapsto (\xi_1,\ldots, \xi_{2n},u,v)\Lambda(t)
$$
representing the action (such transformations of the main generators induce appropriate mappings of $[k_{ij}]$). Note that if $\varphi$ is in $H_N^{u,v,k}$, then the conjugation by $\Lambda(t)$ is also in $H_N^{u,v,k}$, as the action upon $u$ and $v$ is that of $\Lambda(t)\circ \Lambda(t)^{-1}$.

\smallskip

We are going to examine the behavior of such one-parameter families near singularities of $\Lambda(t)$.

Suppose that, as $t$ tends to zero, the $i$-th eigenvalue of $\Lambda(t)$ also tends to zero as $t^{m_i}$, $m_i\in\mathbb{N}$. 

Let $\lbrace m_i,\;i=1,\ldots 2n+2\rbrace$ be the set of degrees of singularity of eigenvalues of $\Lambda(t)$ at zero. Suppose that for every pair $(i,j)$ the following holds: if $m_i\neq m_j$, then there exists a positive integer $M$ such that
$$
\text{either\;\;} m_iM\leq m_j\;\;\text{or\;\;}m_jM\leq m_i.
$$
We will call the largest such $M$ the \textbf{order} of $\Lambda(t)$ at $t=0$. As $m_i$ are all set to be positive integer, the order equals the integer part of $\frac{m_{\text{max}}}{m_{\text{min}}}$.

We now formulate the criterion
\begin{prop}[Singularity trick]\label{singtrick}
An element $\varphi\in\Aut_{u,v,k} P_{n+1,\mathbb{C}}^h[k_{ij}]$ belongs to $H_N^{u,v,k}$ if and only if for every linear matrix curve $\Lambda(t)$ of order $\leq N$ the curve
$$
\Lambda(t)\circ\varphi\circ\Lambda(t)^{-1}
$$
does not have a singularity (a pole) at $t=0$.
\end{prop}
\begin{proof}
Suppose $\varphi\in H_N^{u,v,k}$ and fix a one-parametric family $\Lambda(t)$. Without loss of generality, we may assume that the first $2n$ main generators $\lbrace \xi_1,\ldots, \xi_{2n}\rbrace$ correspond to eigenvectors of $\Lambda(t)$. If $\xi_i$ denotes any of these main generators, then the action of $\Lambda(t)\circ\varphi\circ \Lambda(t)^{-1}$ upon it reads
$$
\Lambda(t)\circ\varphi\circ \Lambda(t)^{-1}(\xi_i) = \xi_i + t^{-m_i}\sum_{l_1+\cdots+l_{2n} = N}a_{l_1\ldots l_{2n}}t^{m_1l_1+\cdots+m_{2n}l_{2n}}P_{i}(\xi_1,\ldots, \xi_{2n},h,k_{ij}) + S_i
$$
where $P_i$ is homogeneous of total degree $N$ (in the previously defined grading) and the height of $S_i$ is greater than $N$. One sees that for any choice of $l_1,\ldots,l_{2n}$ in the sum, the expression
$$
m_1l_1+\cdots+m_{2n}l_{2n} - m_i\geq m_{\text{min}}\sum l_j-m_i=m_{\text{min}}N-m_i\geq 0,
$$
so whenever $t$ goes to zero, the coefficient will not go to infinity. The same argument applies to higher-degree monomials within $S_i$.

The other direction is established by contraposition. Assuming $\varphi\notin H_N^{u,v,k}$, we need to prove the existence of linear curves with suitable eigenvalue behavior near $t=0$ which create singularities via conjugation with the given automorphism.

Suppose first that the image of $\xi_1$ under $\varphi$ possesses a monomial which is not divisible by $\xi_1$ or any $k_{1j}$ ($j\neq 1$). Then one can take $m_1$ and $m_2<m_1$ such that
$$
(N+1)m_2\geq m_1\geq Nm_2
$$
and set the curve $\Lambda(t)$ to be given by a diagonal matrix with entries $t^{m_1},\;t^{m_2},\;t^{m_2},\ldots$. It is easily checked that conjugation of $\varphi$ by this curve creates a pole at the coefficient of the chosen monomial.

The general case can be reduced to this special case by means of transformations of the form ($\lambda$ and $\delta$ are suitable constants)
\begin{gather*}
\xi_1\mapsto \xi_1 + \lambda u + \delta v,\\
k_{ij}\mapsto k_{ij},\;\;1<i,j\leq 2n,\\
k_{1j}\mapsto k_{1j} + \lambda k_{2n+1,j} + \delta k_{2n+2,j},\\
k_{1,2n+1}\mapsto k_{1,2n+1} + \delta k_{2n+2,2n+1},\\
k_{1,2n+2}\mapsto k_{1,2n+2} + \lambda k_{2n+1,2n+2}.
\end{gather*}
Conjugation with these transformations create in the image of $\xi_1$ under the resulting automorphism a monomial from the previous case. In order to obtain the curve $\Lambda(t)$ from the diagonal curve acting on the conjugated automorphism, one needs only conjugate it with the inverse of the above transform. The singularity trick is proved.

\end{proof}
The skew augmented Weyl algebra counterpart of the singularity trick is valid.
\begin{cor}\label{singtrickweyl}
An element $\varphi\in\Aut_{u,v,k} W_{n+1,\mathbb{C}}^h[k_{ij}]$ belongs to $G_N^{u,v,k}$ if and only if for every linear matrix curve $\Lambda(t)$ of order $\leq N$ the curve
$$
\Lambda(t)\circ\varphi\circ\Lambda(t)^{-1}
$$
does not have a singularity at $t=0$.
\end{cor}
The proof of this statement is essentially the same as that of Proposition \ref{singtrick}. Note that thanks to the choice of grading -- the one in which the degree of all $k_{ij}$ is two -- the reordering of the non-commuting variables in a word cannot produce monomials of smaller total degree.

\smallskip

The implementation of the singularity trick in the proof of Proposition \ref{skewthetaprop} requires also the following general fact.
\begin{lem} \label{lem1} Let
$$
\Phi: X\rightarrow Y
$$
be a morphism of affine algebraic sets, and let $\varphi(t)$ be a curve (more simply, a one-parameter family of points) in $X$. Suppose that $\varphi(t)$ does not tend to infinity as $t\rightarrow 0$. Then the image $\Phi \varphi(t)$ under $\Phi$ also does not tend to infinity as $t\rightarrow 0$.
\end{lem}
The proof of the Lemma is an easy exercise and is left to the reader.

Proposition \ref{skewthetaprop} is now an elementary consequence of the above Lemma together with the singularity trick (Proposition \ref{singtrick} and Corollary \ref{singtrickweyl}). Indeed, let us assume the contrary -- i.e. that for some $N$
$$
\Phi^{hk}(G_N^{u,v,k}) \nsubseteq H_N^{u,v,k}.
$$
Then there exists an element $\varphi\in G_N^{u,v,k}$ such that its image $\Phi^{hk}(\varphi)\notin H_N^{u,v,k}$. By Proposition \ref{singtrick}, there is a linear automorphism (matrix) curve $\Lambda(t)$ of order $\leq N$ such that the curve
$$
\Lambda(t)\circ \Phi^{hk}(\varphi)\circ \Lambda(t)^{-1}
$$
has a pole at $t=0$. Since $\Phi^{hk}$ is point-wise stable on linear variable changes, the latter curve is the image under $\Phi^{hk}$ of the curve
$$
\Lambda(t)\circ \varphi\circ \Lambda(t)^{-1}.
$$
By our assumption, $\varphi\in G_N^{u,v,k}$; therefore, by Corollary \ref{singtrickweyl}, the curve above has no singularity at $t=0$. But then the statement that the curve $$\Lambda(t)\circ \Phi^{hk}(\varphi)\circ \Lambda(t)^{-1}$$ -- which is the image of the former curve under the morphism $\Phi^{hk}$ -- has a singularity at $t=0$ yields a contradiction with Lemma \ref{lem1}. Proposition \ref{skewthetaprop} is proved.

The immediate consequence of Proposition \ref{skewthetaprop} is the following result.
\begin{thm}\label{skewcontthm}
The mapping
$$
\Phi^{hk}: \Aut_{u,v,k} W_{n,\mathbb{C}}^h[k_{ij}]\rightarrow \Aut_{u,v,k} P_{n,\mathbb{C}}^h[k_{ij}]
$$
is continuous in the power series topology defined at the start of the section.
\end{thm}
This was the main objective of the singularity trick, and this result will provide the means to establish the canonicity of the symplectomorphism lifting procedure we define in the next subsection.

\subsection{Lifting in the $h$-augmented and skew augmented cases}

We now proceed with the resolution of the symplectomorphism lifting problem for both augmented and skew augmented algebras. More specifically, we will show how the results of the singularity analysis procedure conducted in the previous subsection provide for a way to construct the inverse to the homomorphisms $\Phi^h$ and $\Phi^{hk}$.


\smallskip

Suppose given an automorphism $\varphi\in \Aut P^h_{n,\mathbb{C}}$ of the $h$-augmented Poisson algebra. Without loss of generality, we may assume that the linear part of $\varphi$ is the identity matrix: indeed, one can compose $\varphi$ with tame automorphisms (tame approximation of automorphisms of $P^h_{n,\mathbb{C}}$ is valid according to an argument similar to that of \cite{KGE}), so that the linear part of the resulting automorphism is the identity map; also the morphism $\Phi^h$ is point-wise stable on tame automorphisms.

We add two more $h$-Poisson variables (and lift $\varphi$ to an automorphism of the new algebra by demanding it be stable on the new generators) and, correspondingly, consider the skew Poisson version -- the algebra $P_{n+1,\mathbb{C}}^h[k_{ij}]$ with the last two variables denoted by $u$ and $v$. Our objective is to realise the algebra $P^h_{n,\mathbb{C}}$ as a subalgebra in an appropriate localization of $P_{n+1,\mathbb{C}}^h[k_{ij}]$. To that end, we consider the algebra $P_{n+1,\mathbb{C}}^h[k_{ij}]$  and transform the main generators
$$
\lbrace \xi_1,\ldots, \xi_{2n},u,v\rbrace
$$
to
$$
\lbrace x_1,\ldots, x_{2n},u,v\rbrace
$$
with $\lbrace x_i, u\rbrace = 0$ and $\lbrace x_i, v\rbrace = 0$. The change of the generating set is required to properly define the action of $\varphi$, so that it will be an automorphism and will be in agreement with the conditions of Proposition \ref{skewthetaprop}. The variable change is done according to
$$
x_i = \xi_i - \alpha_i u - \beta_i v
$$
with $\alpha_i = k_{i,2n+2}k_{2n+1,2n+2}^{-1}$ and $\beta_i = - k_{i,2n+1}k_{2n+1,2n+2}^{-1}$ for $i=1,\ldots, 2n$. We extend the coefficient ring by adding the necessary variables. The new generators $\lbrace x_1,\ldots, x_{2n}\rbrace$ commute according to
\begin{gather*}
\lbrace x_i, x_j\rbrace = h(k_{ij} - \alpha_jk_{i,2n+1}+\alpha_ik_{j,2n+1}-\beta_jk_{i,2n+2}+\\
\beta_ik_{j,2n+2}+ (\alpha_i\beta_j - \alpha_j\beta_i)k_{2n+1,2n+2}) = h\tilde{k}_{ij}.
\end{gather*}
Note that the new commutation relation matrix, which we denote by $[\tilde{k}_{ij}]$\footnote{We exclude $u$ and $v$, so that $i$ and $j$ run from $1$ to $2n$.}, is again skew-symmetric, and that its entries are $\mathbb{C}$-polynomial in the entries of the initial matrix and their inverses.

We now reduce the matrix $[\tilde{k}_{ij}]$ to the standard form (corresponding to the algebra $P^h_n$) by transforming $\lbrace x_1,\ldots, x_{2n}\rbrace$ to $\lbrace q_1,\ldots, q_n,p_1,\ldots, p_n\rbrace$ with
$$
\lbrace p_i, q_j \rbrace = h\delta_{ij}.
$$
The new variables $p_i$ and $q_j$ are expressed as linear combinations of $x_1,\ldots, x_{2n}$ with coefficients in the appropriate polynomial ring.

The algebra $P^h_{n,\mathbb{C}}$ is therefore a subalgebra of the algebra generated by $$\lbrace q_1,\ldots, q_n,p_1,\ldots, p_n, u,v\rbrace$$ (together with $h$ as the augmentation variable), as the Poisson bracket takes its proper form after the standard form reduction, while $\mathbb{C}$ is a subring of the coefficient ring.

We extend our automorphism $\varphi$ to act on this algebra: on $p_i$ and $q_j$ its action is given by definition, and we impose $\varphi(u) = u$, $\varphi(v) = v$ and $\varphi(k_{ij}) = k_{ij}$. Thus, starting from $\varphi$ we have arrived at an automorphism $\bar{\varphi}$ of the localized skew Poisson algebra.

With respect to the skew augmented algebra generator set  $\lbrace \xi_1,\ldots, \xi_{2n},u,v\rbrace$ this automorphism is generally not polynomial in $k_{ij}$, although it always will be polynomial in $h$. In order to construct from it an automorphism of the skew algebra, we need to get rid of the denominators first. This is accomplished by the following lemma.
\begin{lem}\label{conjugationlemma}
For every $\bar{\varphi}$ constructed as above, there is a polynomial $P$ in $k_{ij}$, such that conjugation of $\bar{\varphi}$ with the transformation
$$
(\xi_1,\ldots, \xi_{2n},u,v)\mapsto (P\xi_1,\ldots, P\xi_{2n},Pu,Pv),\;\;h\mapsto P^2h
$$
is polynomial in $k_{ij}$. The polynomial $P$ depends only on the two systems of algebra generators.
\end{lem}
\begin{proof}
Indeed, the denominators in the expression for $\bar{\varphi}$ are polynomial in $k_{ij}$ coming from the separation of the $(u,v)$-plane and the standard form reduction (at which point the determinant of $[\tilde{k}_{ij}]$ makes its contribution). One can therefore find appropriate $P(k_{ij})$ to cancel these denominators. Furthermore, the polynomial $P$ depends only on the two generator systems -- more specifically, on the transformation matrix between those systems.
\end{proof}
We denote the result of the conjugation of Lemma \ref{conjugationlemma} by $\varphi^P$. The images of the main generators (both in the cases of the initial -- skew -- generators as well as those which correspond to the standard form) under $\varphi^P$ are, by Lemma \ref{conjugationlemma}, polynomial in $k_{ij}$, and are also by construction polynomial in $h$.

The automorphism $\varphi^P$, when acting upon the standard form generators
$$
\lbrace q_1,\ldots, q_n,p_1,\ldots, p_n\rbrace
$$
can be viewed as an automorphism of the $h$-augmented Poisson algebra $P^h_{n,\mathbb{C}[\lbrace hk_{ij}\rbrace]}$ over the polynomial ring $\mathbb{C}[\lbrace hk_{ij}\rbrace]$. The $\mathbb{Z}$-grading of this algebra is specified by assigning degree $1$ to the main generators and degree $0$ to $h$ and all $k_{ij}$. As an automorphism of this Poisson algebra, $\varphi^P$ admits, by an argument virtually identical to the main result of \cite{KGE}, a tame automorphism (symplectomorphism) sequence converging to it in the power series topology induced by the above grading. Let us fix such a sequence and denote it by $\lbrace \psi_m\rbrace$.

Every element $\psi_k$ of the tame sequence is such that the images under $\psi_m$ of the generators $\lbrace q_1,\ldots, q_n,p_1,\ldots, p_n\rbrace$ are polynomial in $h$ and $k_{ij}$. Importantly, the tame sequence $\lbrace \psi_m\rbrace$ converging to the $\mathbb{C}[\lbrace hk_{ij}\rbrace]$-automorphism $\varphi^P$ can be connected to the original $h$-augmented symplectomorphism $\varphi$ by inversion of the procedure which leads to the definition of $\varphi^P$. Precisely, we have the following statement.
\begin{prop}\label{goodbehaviorprop}
For every $\varphi$ and every sequence $\lbrace \psi_m\rbrace$ converging to $\varphi^P$ as above, there is a sequence $\lbrace \sigma_m\rbrace$ of tame $\mathbb{C}$-symplectomorphisms of $P_{n,\mathbb{C}}$ converging to $\varphi$ with respect to the topologies with $\Deg h = 0$ and $\Deg h = 2$.
\end{prop}
\begin{proof}
The sequence $\lbrace \sigma_m\rbrace$ is constructed from $\lbrace\psi_m\rbrace$ by reversing the conjugation by $P$ and disposing of the stable variables $u$ and $v$. We note that the conjugation is a group homomorphism, which means that it suffices to prove that the reverse conjugation disposes of $k_{ij}$ in every elementary tame automorphism (as $\psi_m$ are composition of elementary automorphisms). The latter property, however, is obvious.

The convergence of $\lbrace \sigma_m\rbrace$ follows immediately from Lemma \ref{conjugationlemma} and Lemma \ref{convergencelemma} below.
\end{proof}

When acting upon the localized skew Poisson algebra generators $\lbrace \xi_1,\ldots, \xi_{2n}\rbrace$, the augmented symplectomorphisms $\psi_m$ need not be polynomial in $k_{ij}$, and therefore $\psi_m$ are not in general images of automorphisms of the skew Poisson algebra under localization. This is remedied by application of Lemma \ref{conjugationlemma}: one can find a polynomial $P_1$ in the variables $k_{ij}$, such that the conjugation of every element $\psi_m$ of the tame sequence with the mapping
$$
(\xi_1,\ldots, \xi_{2n},u,v)\mapsto (P_1\xi_1,\ldots, P_1\xi_{2n},P_1u,P_1v),\;\;h\mapsto P_1^2h
$$
yields an automorphism of the localized skew Poisson algebra polynomial in $k_{ij}$. Again, as in Proposition \ref{goodbehaviorprop}, one can return to a sequence of tame symplectomorphisms of $P_{n,\mathbb{C}}$ by reversing the conjugation.

We then have the following statement.
\begin{lem}\label{convergencelemma}
The sequence $\lbrace \psi^{P_1}_m\rbrace$ converges to the conjugated automorphism $(\varphi^P)^{P_1}$ in the power series topology with $\Deg h = \Deg k_{ij} = 0$ as well as in the power series topology with $\Deg h = 0$, $\Deg k_{ij} = 2$.
\end{lem}
\begin{proof}
The first half of the statement follows from the construction of the tame sequence and from the observation that, due to the fact that the two coordinate systems are connected by a transformation that has zero free term, the height of the polynomials $P$ and $P_1$ is at least one.

One then obtains convergence in the power series topology relevant to the singularity trick (Proposition \ref{skewthetaprop}) from that in the approximation power series topology by noting that giving a non-zero degree to $k_{ij}$ may only make the consecutive approximations closer to the limit in the corresponding metric.

\end{proof}

The sequence $\lbrace \psi^{P_1}_m\rbrace$ can, due to its polynomial character with respect to $k_{ij}$, be thought of as a sequence of tame automorphisms of the skew Poisson algebra $P_{n+1,\mathbb{C}}^h[k_{ij}]$ over $\mathbb{C}$ converging to $(\varphi^P)^{P_1}$. We now take the pre-image of the sequence $\lbrace \psi^{P_1}_m\rbrace$ under the morphism $\Phi^{hk}$ to obtain the sequence
$$
\lbrace \sigma'_m\rbrace=\lbrace (\Phi^{hk})^{-1}(\psi^{P_1}_m)\rbrace
$$
of automorphisms of the skew augmented Weyl algebra $W_{n+1,\mathbb{C}}^h[k_{ij}]$. We now may take the formal limit of this sequence, which is ostensibly dependent on the choice of the convergent tame sequence $\lbrace \psi_m\rbrace$, apart from the point $\varphi$ itself. This limit, which we denote by
$$
\Theta^h_P(\varphi, \lbrace \psi_m\rbrace)
$$
to reflect the dependence on the sequence, is given by formal power series in the skew augmented Weyl generators. Applying the inverse to the conjugations performed earlier and disposing of the stable variables (whose presence is justified by the form of the singularity trick and is therefore needed in the proof of independence of the choice of the convergent sequence, as we shall see below), we arrive at a vector of formal power series (the entries of which correspond to images of the generators) in the generators of the $h$-augmented Weyl algebra $W^h_{n,\mathbb{C}}$.

\smallskip

The most important consequence of Theorem \ref{skewcontthm} is the independence of the lifted sequence's formal limit of the choice of the approximating tame sequence $\lbrace \psi_m\rbrace$. We have the following proposition.

\begin{prop} \label{propcanon}
Let $\varphi$ be an automorphism of $P^h_{n,\mathbb{C}}$ and let
$$
\psi_1,\;\ldots,\; \psi_m,\;\ldots
$$
and
$$
\psi'_1,\;\ldots,\; \psi'_m,\;\ldots
$$
be two sequences of tame automorphisms which converge to $\varphi^P$ as in the construction above. Then the lifted sequences
$$
\lbrace (\Phi^{hk})^{-1}(\psi^{P_1}_m)\rbrace\;\;\text{and}\;\;\lbrace (\Phi^{hk})^{-1}(\psi'^{P_1}_m)\rbrace
$$
converge to the same automorphism of the power series completion of the skew augmented Weyl algebra. This means that one must have
$$
\left((\Phi^{hk})^{-1}(\psi^{P_1}_m)\right)^{-1}\circ (\Phi^{hk})^{-1}(\psi'^{P_1}_m)\equiv \Id\;(\text{mod}\;I^{N(k)})
$$
with $N(k)\rightarrow \infty$ as $k\rightarrow \infty$.
\end{prop}
\begin{proof}
This result follows immediately from the continuity of $\Phi^{hk}$ established in the previous subsection.
\end{proof}

The meaning of this Proposition to the symplectomorphism lifting problem is clear: in our construction, the formal limit
$$
\Theta^h_P(\varphi, \lbrace \psi_m\rbrace)
$$
is independent of $\lbrace \psi_m\rbrace$ and is therefore a well-defined function of the point $\varphi$. Furthermore, as can be inferred directly from the tame approximation, this function is homomorphic -- it preserves the group structure given by composition of automorphisms. As the conjugations are also homomorphisms, we conclude that $h$-augmented symplectomorphisms $\varphi \in \Aut P^h_{n,\mathbb{C}}$ are lifted homomorphically to endomorphisms of the power series completion $\hat{W}^h_{n,\mathbb{C}}$ of the $h$-augmented Weyl algebra. For an augmented symplectomorphism $\varphi$, we denote its image with respect to the lifting map by
$$
\Theta^h(\varphi).
$$

Our next objective is to demonstrate that for every symplectomorphism $\varphi$, the image $\Theta^h(\varphi)$ under the lifting map is in fact an automorphism of the $h$-augmented Weyl algebra. In fact, it remains to show only that the generator images with respect to $\Theta^h(\varphi)$ cannot be given by infinite series: indeed, that would imply that $\Theta^h(\varphi)$ is an $h$-augmented Weyl endomorphism; the invertibility of the lifted mapping follows from the canonicity of lifting: indeed, $\Theta^h$ not only preserves compositions but also maps inverses to inverses, therefore for any symplectomorphism $\varphi$ the mapping $\Theta^h(\varphi^{-1})$ will be the inverse of $\Theta^h(\varphi)$.

Alternatively, one can arrive at the invertibility after one shows the polynomial character of the lifted endomorphism: it is known \cite{Tsu2, BKK2} that the direct morphism $\Phi$ -- and hence, by a straightforward extension of the argument in the aforementioned work, its augmented analogue $\Phi^h$ -- distinguishes automorphisms, i.e. the image of a non-automorphism cannot be an automorphism.

\smallskip

The main properties of the mapping $\Theta^h$ can now be summarized in the following way.
\begin{prop}\label{thetaproperties}

1. There exists a well-defined mapping $\Theta^h$ whose domain is $\Aut P^h_{n,\mathbb{C}}$ and whose codomain lies in the set of automorphisms of the power series completion of $W^h_{n,\mathbb{C}}$.

2. $\Theta^h$ is a group homomorphism.

3. For a fixed $\varphi$, the coordinates of $\Theta^h(\varphi)$ -- i.e. the coefficients with respect to the fixed generator basis decomposition -- are given by polynomials in the coordinates of $\varphi$.

4. The skew augmented analogue $\Theta^{hk}$ of $\Theta^{h}$ is continuous in the power series topology.

\end{prop}

\begin{proof}
The first two statements follow immediately from the construction. The third statement follows from the fact that the lifting is independent of the approximating sequence: indeed, that implies that the coefficients of the lifted limit are read off any valid approximating sequence, or more precisely its finite subset (consisting of first several elements, as the limit symplectomorphism is polynomial). But then the coefficients are polynomial in the coordinates of the lifted tame elements, and since only a finite number of them suffices, they are also polynomial in the coordinates of the initial symplectomorphism.

The continuity of $\Theta^{hk}$ is established in a manner identical to that of $\Phi^{hk}$ -- namely with the help of Proposition \ref{singtrick} a proof similar to that of Theorem \ref{skewcontthm} can be executed. Note that it follows from the third statement of this Proposition that $\Theta^{hk}$ fulfills the conditions of Lemma \ref{lem1}, therefore by combining the previously proved statements with the analogous steps for the lifting map, one can show that
$$
\Phi^{hk}(G_N^{u,v,k}) = H_N^{u,v,k}.
$$
\end{proof}

\subsection{The lifted limit is polynomial}
We proceed with establishing the polynomial character of the image $\Theta^h(\varphi)$.
\begin{thm}\label{thmpolynomial}
Let
$$
\Theta^h: \Aut P^h_{n,\mathbb{C}}\rightarrow \Aut \hat{W}^h_{n,\mathbb{C}}
$$
be the lifting homomorphism constructed in the previous section and let, as before, $x_1,\ldots, x_n, d_1,\ldots, d_n,h$ denote the generators of $W^h_{n,\mathbb{C}}$ together with its deformation parameter. Then, for every augmented symplectomorphism $\varphi\in \Aut P^h_{n,\mathbb{C}}$, the images
$$
\Theta^h(\varphi)(x_1),\ldots, \Theta^h(\varphi)(d_n)
$$
are polynomials in $x_i$, $d_i$ and $h$.
\end{thm}
\begin{proof}
Suppose that, contrary to the statement of the theorem, for a fixed $\varphi$ there is an index $i$ such that, say,\footnote{The case for $d_i$ is processed analogously.} $\Theta^h(\varphi)(x_i)$ is a true infinite series of Weyl monomials.

Let $\lambda$ be a parameter and let
$$
\tau_\lambda: (x_1,\ldots, d_n, h)\mapsto (\lambda x_1,\ldots,\lambda  d_n,\lambda^2 h)
$$
denote the family of dilation transformations parameterized by $\lambda$. For fixed $\varphi$, define
$$
\varphi_\lambda = \tau_\lambda^{-1}\circ\varphi\circ \tau_\lambda
$$
to be the parametric family of $h$-augmented symplectomorphisms constructed by conjugating $\varphi$ with the dilations.

We introduce a pair of auxiliary variables $u$ and $v$, $\lbrace v,u\rbrace = h$ (their Weyl counterparts will be also denoted by $u$ and $v$ -- obviously it does not create any ambiguities) and, for a fixed large enough positive integer $k$, define the following parametric family of linear transformations
$$
\psi_\lambda: u\mapsto u + \lambda^kx_i,\;\;p_i\mapsto p_i - \lambda^k v
$$
(all other generators are unchanged). As always, we extend the action of $\varphi$ to the auxiliary variables by setting $\varphi(u) = u$ and $\varphi(v) = v$ (while the dilation extends to $(u,v)\mapsto (\lambda u, \lambda v)$). Consider the following parametric family of $h$-augmented symplectomorphisms:
$$
\varphi_{t,\lambda}= \varphi_\lambda\circ \psi_\lambda\circ\varphi_{\lambda}^{-1}.
$$
The conjugation of $\varphi$ by the inverse to the dilation $\tau_\lambda$ amounts to multiplying each homogeneous component of degree $m$ by $\lambda^{1-m}$. Therefore (as $\varphi$ is polynomial) for large enough $k$, the curve $\varphi_{t,\lambda}$ can be continuously extended by its limit at $\lambda = 0$ -- namely, by the identity symplectomorphism. Continuity is understood in the sense of continuous dependence of coordinates of the symplectomorphism on the parameter.

Now, the image of $u$ under the lifted curve $\Theta^h(\varphi_{t,\lambda})$ is
$$
u + \lambda^k\Theta^h(\varphi_\lambda)(x_i)
$$
and, as by assumption $\Theta^h(\varphi)(x_i)$ is an infinite series, is itself an infinite series. But since $\Theta^h$ is identical on linear transformations, we have
$$
\Theta^h(\varphi_\lambda) = \tau_\lambda^{-1}\circ\Theta^h(\varphi) \circ \tau_\lambda
$$
from which it follows that in the image $\Theta^h(\varphi_{t,\lambda})(u)$ there will be monomials with coefficients proportional to $\lambda^{-m}$ for $m$ greater than any fixed arbitrary natural number.

However, it follows from the third statement in Proposition \ref{thetaproperties} that the coordinates (i.e. coefficients of Weyl monomials) of the lifted symplectomorphism are continuous functions of the coordinates of the symplectomorphism, therefore since the curve $\varphi_{t,\lambda}$ is regular at $\lambda = 0$, then so must also be its image under $\Theta^h$ -- a contradiction.
\end{proof}

We can now combine this theorem with the results of the previous subsection.
\begin{thm} \label{inversethm}
The lifting homomorphism $\Theta^h$ is the inverse to the direct homomorphism $\Phi^h$.
\end{thm}
\begin{proof}
Indeed, Theorem \ref{thmpolynomial} shows that the compositions $\Phi^h\circ\Theta^h$ and $\Theta^h\circ\Phi^h$ are well defined. In order to prove that, say, $\Phi^h\circ\Theta^h=\Id$ one changes the basis of generators (and handles the extension of the base ring as in the prequel) to that of the skew augmented algebra and uses the fact that $\Phi^{hk}\circ\Theta^{hk}$ coincides with the identity map on the dense subset of tame symplectomorphisms and hence must be the identity map everywhere (the spaces in question are metric spaces, in particular they are Hausdorff).
\end{proof}
This theorem has one important corollary.
\begin{cor}\label{correduction}
The lifting map $\Theta^h$ and the direct map $\Phi^h$ are consistent with modulo infinite prime reductions. That is, for any symplectomorphism $\varphi$, almost all its modulo $p_m$ ($m$ in the index set in the ultraproduct decomposition) reductions coincide with the (twisted by inverse Frobenius) restrictions to the center of the modulo $p_m$ reductions of its lifting $\Theta^h(\varphi)$.
\end{cor}
\begin{proof}
This statement is essentially a reformulation of Theorem \ref{inversethm}, if one takes into account the construction of $\Phi^h$ -- the image of the Weyl algebra automorphism being reconstructed from the ultraproduct of positive characteristic automorphisms restricted to the center.

\end{proof}

\begin{remark}
Alternatively, one could come up with a line of reasoning more conforming to the combinatorial side of the constructions employed thus far. If there were a symplectomorphism $\varphi$ which, after lifting and subsequent direct homomorphism action (ultraproduct decomposition followed by restriction to the center) produces a different symplectomorphism $\varphi'$, then one may take a sequence of elementary symplectomorphism whose total action on $\varphi$ maps it to an element which is the identity map modulo terms of degree $N$. If $\varphi'\neq \varphi$, then the action of the same elementary sequence on $\varphi'$ will produce an element which admits a term of degree $1<M<N$. The two objects are then mapped to automorphisms of the skew augmented algebra and lifted (we also note that the mapping induced by the change of basis is continuous, as it is in essence a dilation of generators by a polynomial of height at least one).
By the singularity trick (Proposition \ref{singtrick}) such an object can be conjugated by an appropriate linear variable change in order to produce a singular curve. Now, by construction, the skew version of $\varphi$ lifts to a skew Weyl automorphism, and again by the singularity trick (essentially by the continuity of $\Theta^{hk}$) the lifting of the partially approximated automorphism (i.e. after the action of the elementary automorphisms) cannot have a singularity of order $\leq N$. However, as the skew version of $\varphi$ and $\varphi'$ (as well as its partial approximation) corresponds in the ultraproduct to restrictions to the center, the restriction of an object which is not singular of order $\leq N$ must also be non-singular of order $\leq N$, in contradiction with the existence of $\varphi'$.
\end{remark}

The $h$-augmented counterpart of the Kontsevich conjecture follows at once from Theorem \ref{inversethm}.
\begin{thm}\label{mainthmh}
The homomorphism
$$
\Phi^h :\Aut W^h_{n,\mathbb{C}}\rightarrow \Aut P^h_{n,\mathbb{C}}
$$
is an isomorphism.
\end{thm}

As the map $\Phi^h$ is closely related to the morphism $\Phi^{hk}$ of the skew augmented case, and, correspondingly, as $\Theta^h$ is related to the lifting map for the skew augmented symplectomorphisms (essentially given by $\Theta^h_P$), we obtain another important consequence of Theorem \ref{thmpolynomial}.
\begin{thm}\label{mainthmskew}
Let $\Aut_{k}P_{n,\mathbb{C}}^h[k_{ij}]$ and $\Aut_{k}W_{n,\mathbb{C}}^h[k_{ij}]$ denote the automorphism subgroups of the skew Poisson and Weyl algebras consisting of those automorphisms that map $k_{ij}$ to $\mathbb{C}$-linear combinations of $k_{ij}$. Then the mapping
$$
\Phi^{hk}:\Aut_{k}W_{n,\mathbb{C}}^h[k_{ij}]\rightarrow \Aut_{k}P_{n,\mathbb{C}}^h[k_{ij}]
$$
is an isomorphism.
\end{thm}
Theorem \ref{mainthmskew} is significant to the proof of the independence of the morphism $\Phi$ of the choice of infinite prime given in \cite{K-BE2}.

\subsection{Specialization}

Now that we have established the isomorphism between the automorphism groups of the $h$-augmented Weyl and Poisson algebras, the proof of the Main Theorem reduces to specializing to $h=1$. That, however, is by no means a trivial affair, as one needs to take care not only of the automorphisms polynomial in $h$ (for which the existence of lifting has been established), but also of those which are polynomial in $h^{-1}$.

The necessity of extension of the domain of the lifting map can be seen from the following argument.
Suppose $\varphi^h$ is an automorphism of the $h$-augmented algebra $P^h_{n,\mathbb{C}}$ which acts as the identity map on $h$. Since it is stable on $h$, it corresponds to an automorphism of the $\mathbb{C}[h]$-algebra (where $h$ is a parameter and not a generator, which the can be effectively adjoined to the ground field) generated by $x_i,\;p_j$ with the Poisson bracket containing $h$. This object, after appropriate localization, maps to an automorphism of the (augmented) Poisson algebra $P^h_n$ with the ground field $\mathbb{C}(h)$. On the other hand, any automorphism $\varphi$ of $P_{n,\mathbb{C}}$ can be made into a $\mathbb{C}(h)$-automorphism $\varphi^h$ by introducing a scalar $h$ and conjugating $\varphi$ with a mapping
$$
x_i' = hx_i,\;\;p_j' = p_j.
$$
The resulting transformation will be an automorphism of the Poisson $\mathbb{C}(h)$-algebra with the bracket as in the augmented algebra $P^h_n$,\emph{ however in general the images of the generators under this automorphism will contain negative powers of $h$.  }
Its specialization to $h=1$ returns it to $\varphi$. Therefore, every polynomial symplectomorphism has a pre-image under specialization of the $\mathbb{C}(h)$-algebra automorphisms. The conclusion is that Theorem \ref{mainthmh} does not immediately imply that $\Phi$ is an isomorphism; rather, the domain of the lifting map $\Theta^h$ needs to be extended to the points with rational dependency on the augmentation parameter, at which point the claim that $\Phi$ has an inverse given by the specialization of the extended lifting map $\Theta^h$ becomes valid.

\smallskip

The extension of the domain is accomplished in the following way. For a symplectomorphism $\varphi$ which is rational in $h$, we will construct images $\Theta^h(\varphi)(x_i)$ and $\Theta^h(\varphi)(d_i)$ one by one by introducing auxiliary variables and twisting the symplectomorphism in order to create an object polynomial in $h$ -- using the fact that the action of $\Theta^h$ is well defined -- from which the form of the corresponding lifted generator image may be extracted. As the procedure yields not all of the images simultaneously, we will need to check its canonical nature as well as verify the commutation relations.

We fix $i$, $1\leq i\leq n$, which corresponds to the image of $x_i$, and introduce a pair $u$, $v$ of auxiliary variables which are extra $x$ and $p$ with respect to the augmented Poisson bracket. We also add the corresponding augmented Weyl variables which we denote by $\hat{u}$ and $\hat{v}$. Let
$$
\lambda = h^k
$$
be $k$-th power of the augmentation parameter, for large enough $k$. Define the automorphism
$$
\psi_{\lambda}: u\mapsto u + \lambda x_i,\;\;p_i\mapsto p_i - \lambda v.
$$
We extend $\varphi$ to the new algebra by its identical action on the auxiliary variables and denote the extended map by $\varphi_a$. Consider the following twisted automorphism:
$$
\varphi_{t,\lambda} = \varphi_a\circ\psi_{\lambda}\circ\varphi_a^{-1}.
$$
As $k$ can be taken arbitrarily large, the mapping $\varphi_{t,\lambda}$ will be polynomial in $h$ for all $k>k_0$ (where $k_0$ depends on $\varphi$ but is finite for the fixed automorphism). We can now read off the expression for the image of $x_i$ under $\varphi$ from the action of $\varphi_{t,\lambda}$ on the auxiliary variable $u$:
$$
\varphi_{t,\lambda}(u) = u + h^k\varphi(x_i);
$$
the expression is polynomial in $h$, and $\varphi_{t,\lambda}$ thus admits lifting to an automorphism of the $h$-augmented Weyl algebra $W^h_{n,\mathbb{C}}$. As we will show in a moment, the action of the lifted automorphism on $\hat{u}$ will be given by the expression
$$
\hat{u} + h^k P_i(x_1,\ldots, d_n,h)
$$
($P_i$ is polynomial in $x_1,\ldots, d_n$ and rational -- or, more precisely, Laurent-polynomial -- in $h$) so that one can set
$$
\hat{\varphi}(x_i) = P_i(x_1,\ldots, d_n,h)
$$
and thus, for all $x_i$, obtain the action of the lifted symplectomorphism. Switching the roles of $x_i$ and $d_i$ allows for reconstruction of the images of $d_i$. As a result, we get a mapping
$$
\hat{\varphi}: (x_1,\ldots, d_n)\mapsto (P_1(x_1,\ldots, d_n,h),\ldots, Q_n(x_1,\ldots, d_n,h)).
$$
Note, however, that as the lifting is not defined for the components in the composition (with the exception of $\psi_{\lambda}$), one cannot immediately conclude that the image of $\hat{u}$ under the lifted map will be of the form as above, or that the parts which depend on $x_1,\ldots, d_n$ will combine to a well defined automorphism -- these properties need to be verified.

\smallskip

The first step is to ensure the constructed mapping $\hat{\varphi}$ is well defined (is canonical with respect to $\varphi$). This property in fact follows from the consistency with modulo infinite prime reductions given by Corollary \ref{correduction}, and is manifested in the form of the next two lemmas.

\begin{lem}\label{twistedliftcanon}
Suppose $\theta$ is an $h$-augmented polynomial symplectomorphism over $\mathbb{C}$. Denote by $\lbrace \theta_p\rbrace$ the sequence of characteristic $p$ symplectomorphisms representing its modulo $[p]$ reduction. For a generic element $p$ in a sequence representing $[p]$, denote the Weyl generators by $x_1,\ldots, x_n, d_1,\ldots, d_n$ and the corresponding $p$-th powers generating the center of the Weyl algebra over $\mathbb{F}_p$ by  $\xi_1,\ldots, \xi_n, \eta_1,\ldots, \eta_n$.
Then, for almost all $p$ in $[p]$ (in the sense of the ultrafilter), the image under $\theta_p$ of every central generator admits a unique pre-image Weyl polynomial $\hat{H}$ with respect to taking the $p$-th power and pulling back the coefficients by the inverse Frobenius automorphism.
\end{lem}
\begin{proof}
We prove the statement for $H = \theta_p(\xi_i)$ -- the case $\eta_j$ is identical.

Suppose first that
$$
\theta_p(\xi_i) = \xi_i = x_i^p.
$$
Then the Newton polyhedron of the image $\theta_p(\xi_i)$ has only one vertex, therefore -- as taking the $p$-th power only dilates the Newton polyhedron -- the polynomial $\hat{\theta}_p(x_i)$ must be equal to $x_i$.

The general case uses Corollary \ref{correduction}, which states that modulo $[p]$ reductions of $\theta$ and its lifting $\hat{\theta}$ are consistent -- that is, for almost all $p$ in $[p]$, the restriction of $\hat{\theta}_p$ to the center (twisted by the inverse Frobenius acting on the coefficients) coincides with $\theta_p$. The application is as follows. Suppose
$$
H = \theta_p(\xi_i)
$$
is the image of $\xi_i$. From Corollary \ref{correduction} we know that
$$
H = \Fr_*^{-1}\hat{\theta}_p(x_i^p)
$$
where $\Fr_*^{-1}$ is the action of the inverse Frobenius automorphism on the coefficients of the polynomial. The last equation is equivalent to
$$
\hat{\theta}_p^{-1}(\Fr_*(H)) = x_i^p.
$$
By the special case above, there exists a unique Weyl polynomial $\hat{G}$ such that
$$
\hat{G}^p = \hat{\theta}_p^{-1}(\Fr_*(H)).
$$
But then
$$
H = \Fr_*^{-1}(\hat{\theta}_p(\hat{G}^p))
$$
which is exactly what we wanted.
\end{proof}

It will be convenient to denote the one-to-one correspondence between modulo $p$ reductions of central polynomials coming from characteristic zero symplectomorphisms with their Weyl liftings by $\Phi^h_p$ (for this correspondence is, as evidenced by Lemma \ref{twistedliftcanon}, shares essential nature with the characteristic zero direct homomorphism $\Phi^h$).

We now apply the above lemma in order to establish the form of the pre-image Weyl polynomial in the case of auxiliary variables $u,v$ and the central polynomial of a special type.
\begin{lem}\label{formlemma}
Let $u,v$ denote the extra Poisson variables, and let
$$
H = u + h^k\varphi(x_i)
$$
be the image of $u$ under the twisted automorphism coming from $\varphi$ as above ($\varphi(x_i)$ is rational in $h$ but $h^k\varphi(x_i)$ is polynomial in $h$). Then the unique pre-image $\hat{H}$ of $H$ with respect to the correspondence $\Phi^h_p$ of the previous lemma has the form
$$
\hat{H} = \hat{u} + h^k P_i(x_1,\ldots, d_n,h)
$$
where $P_i$ is rational in $h$.
\end{lem}
\begin{proof}
We establish the statement in several elementary steps. Firstly, as $H$ does not contain the auxiliary variable $v$, $\hat{H}$ does not contain its Weyl counterpart $\hat{v}$: indeed, otherwise the Newton polyhedron of $H$ would contain (in the case of $v$ carrying great enough weight to make the corresponding monomial the highest-order term) a vertex corresponding to the monomial containing $\hat{v}$. \footnote{Note that $\Phi^h_p$ behaves toward the Newton polyhedra as the homomorphism taking the $p$-th power does.}

Now let
$$
\hat{H} = Q(\hat{u}) + R
$$
where every monomial in $R$ is proportional to generators other than $u$. Then
$$
H = \Phi^h_p(\hat{H}) = \Phi^h_p(Q) + \Phi^h_p(R),
$$
as the two differential operators $Q$ and $R$ commute with each other and therefore taking the $p$-th power is executed as in the commutative case. By Lemma \ref{twistedliftcanon}, we must have
$$
Q(\hat{u}) = \hat{u}.
$$

Finally, we show that if $\hat{H}$ contains monomials which are products of $\hat{u}$ with other generators, then  $\Phi^h_p(\hat{H})\neq H$.  Indeed, if such a monomial had a non-zero coefficient in  $\hat{H}$, then there would exist a grading under which this monomial would be the highest-order term (corresponding to a vertex in the Newton polyhedron). Then the image $\Phi^h_p(\hat{H})$ would also have a monomial corresponding to this highest-order term with non-zero coefficient, as taking the $p$-th power dilates the polyhedron and therefore maps the extremal points to extremal points.

The conclusion is that the polynomial $\hat{H}$ has the form
$$
\hat{u} + \tilde{P}_i(x_1,\ldots, d_n,h).
$$
Taking out $h^k$ from $\tilde{P}_i$ leaves us with the form we needed.
\end{proof}

Lemma \ref{formlemma} provides a canonical way to relate the images $\varphi(x_i)$ and $\varphi(d_j)$ of the initial symplectomorphism with the Weyl pre-images. Therefore, as an array of differential operators, the lifting $\hat{\varphi}$ is well defined. We denote the polynomials in Weyl generators in the correspondence by
$$
(\hat{\varphi}(x_1),\ldots, \hat{\varphi}(d_n)).
$$

We now need to verify the commutation relations in order to establish its homomorphic character. Again we have two lemmas.
\begin{lem}\label{commutationlemma1}
\begin{gather*}
[\hat{\varphi}(x_i), \hat{\varphi}(x_j)] = [\hat{\varphi}(d_i), \hat{\varphi}(d_j)]=0,\\
[\hat{\varphi}(x_i), \hat{\varphi}(d_j)]=0,\;\;i\neq j.
\end{gather*}
\end{lem}
\begin{proof}
It suffices to prove
$$
[\hat{\varphi}(x_1), \hat{\varphi}(x_2)]=0
$$
thanks to the variable re-labelling and the existence of the "Fourier transform" -- the automorphism
$$
x_i\mapsto d_i, \;\;d_i\mapsto -x_i.
$$
We introduce two pairs of auxiliary Poisson variables,  $u_1, u_2, v_1, v_2$, and for $\lambda = h^k$ and $k$ large enough consider the automorphism $\psi$:
\begin{gather*}
u_1\mapsto u_1 + \lambda x_1,\;\;u_2\mapsto u_2+\lambda x_2\\
p_1\mapsto p_1-\lambda v_1,\;\;p_2\mapsto p_2-\lambda v_2
\end{gather*}
($\psi$ acts an the identity map on the rest of the generators).

We take the twisted automorphism
$$
\varphi_{t,\lambda} = \varphi_a\circ\psi\circ\varphi_a^{-1}
$$
with $\psi$ now being the chosen linear transformation and take $k$ to be large enough so that the twisted automorphism is a polynomial symplectomorphism. We then lift it with $\Theta^h$ to the $h$-augmented Weyl algebra as before.

By Lemma \ref{formlemma}, the images of the Weyl counterparts $\hat{u}_i$ ($i=1,2$) of $u_1, u_2$ under the lifted twisted automorphism will have the form
$$
\hat{u}_i + \lambda T_i,
$$
the polynomials $T_i$ do not contain the auxiliary variables and $\Phi^h_p(T_i) = \varphi(x_i),\;\;i=1,2$.

Now, as $\varphi_{t,\lambda}$ and its lifting are automorphisms, we must have
$$
[\hat{u}_1 + \lambda T_1,\hat{u}_2 + \lambda T_2]=0
$$
so that
$$
[\hat{u}_1, \hat{u}_2] + \lambda([\hat{u}_1, T_2] + [T_1, \hat{u}_2]) + \lambda^2 [T_1,T_2] = 0
$$
from which it follows immediately that
$$
[T_1,T_2] = 0
$$
as desired.
\end{proof}

\begin{lem}\label{commutationlemma2}
$$[\hat{\varphi}(d_i), \hat{\varphi}(x_i)] = h.$$
\end{lem}
\begin{proof}
We proceed in an manner analogous to the previous lemma: we construct the appropriate twisting from whose lifting the relevant images may be read off and then evaluate the commutator.

Let $u,v$ be auxiliary Poisson variables and let
$$
\psi_1: u\mapsto u+\lambda x_i,\;\;p_i\mapsto p_i - \lambda v,
$$
$$
\psi_2: v\mapsto v+\mu p_i,\;\; x_i\mapsto x_i - \mu u
$$
(in both cases the other generators are mapped to themselves).
Consider the composition
$$
\theta = \psi_1\circ\psi_2.
$$
Then
$$
\theta(u) = u+\lambda x_i,\;\; \theta(v) = v+\mu p_i - \lambda\mu v
$$
and
$$
\theta(x_i) = x_i - \mu u - \lambda\mu x_i,\;\;\theta(p_i) = p_i - \lambda v.
$$
Take
$$
\varphi_{t,\lambda\mu} = \varphi_a\circ\theta\circ\varphi_a^{-1}
$$
where as before $\varphi_a$ extends from $\varphi$ by the identical action on $u,v$. The images of $u,v$ under $\varphi_{t,\lambda\mu}$ read:
$$
\varphi_{t,\lambda\mu}(u) = u + \lambda \varphi(x_i),\;\;\varphi_{t,\lambda\mu}(v) = (1-\lambda\mu)v +\mu\varphi(p_i).
$$
By properly selecting $\lambda$ and $\mu$ as polynomials in $h$, we can make $\varphi_{t,\lambda\mu}$ into a polynomial $h$-augmented symplectomorphism and therefore lift it with $\Theta^h$. Again, by Lemma \ref{formlemma}, the action of the lifted automorphism on $\hat{u}$ and $\hat{v}$ will have the needed form (with the part dependent on $x_i$, $d_j$ given by the images under $\hat{\varphi}$). Now, the commutator of the images of $\hat{u}$ and $\hat{v}$ must be equal to $h$. We therefore have the following:
$$
h = [(1-\lambda\mu)\hat{v} + \mu\hat{\varphi}(d_i), \hat{u}+\lambda\hat{\varphi}(x_i)]=h(1-\lambda\mu) + \lambda\mu[\hat{\varphi}(d_i),\hat{\varphi}(x_i)]
$$
from which the statement follows directly.
\end{proof}

The conclusion is that augmented symplectomorphisms rational in $h$ are lifted to endomorphisms of the augmented Weyl algebra (also rational in $h$) by a homomorphism whose restriction to points polynomial in $h$ coincides with $\Theta^h$. What remains to show is that the lifted mappings are automorphisms, however, this is accomplished by an argument similar to that for points polynomial in $h$ (cf. discussion immediately preceding Proposition \ref{thetaproperties}). Also, thanks to Lemma \ref{formlemma}, we known that the lifting of points rational in $h$ is also the inverse mapping to the extension to these points of the direct homomorphism $\Phi^h$. The specialization to $h=1$ may now be safely executed, and the Main Theorem follows.

\end{document}